\documentclass[11pt, reqno]{article}

\usepackage{amsmath,amssymb,amsfonts,amsthm}
\usepackage{color}

\definecolor{emerald}{rgb}{0.31, 0.78, 0.47}
\definecolor{bleudefrance}{rgb}{0.19, 0.55, 0.91}
\definecolor{brandeisblue}{rgb}{0.0, 0.44, 1.0}
\definecolor{forestgreen}{rgb}{0.13, 0.55, 0.13}
\usepackage{hyperref}
\hypersetup{colorlinks=true,linkcolor=brandeisblue,citecolor=emerald}
\usepackage{cleveref}
\usepackage{caption}
\usepackage{thmtools}
\usepackage{mathrsfs}

\crefname{theorem}{Theorem}{Theorems}
\crefname{thm}{Theorem}{Theorems}
\crefname{lemma}{Lemma}{Lemmas}
\crefname{lem}{Lemma}{Lemmas}
\crefname{remark}{Remark}{Remarks}
\crefname{prop}{Proposition}{Propositions}
\crefname{defn}{Definition}{Definitions}
\crefname{corollary}{Corollary}{Corollaries}
\crefname{conjecture}{Conjecture}{Conjectures}
\crefname{question}{Question}{Questions}
\crefname{chapter}{Chapter}{Chapters}
\crefname{section}{Section}{Sections}
\crefname{figure}{Figure}{Figures}
\crefname{example}{Example}{Examples}

\theoremstyle{plain}
\newtheorem{thm}{Theorem}[section]

\newtheorem{lem}[thm]{Lemma}

\newtheorem{prop}[thm]{Proposition}

\newtheorem*{clm*}{Claim}

\theoremstyle{definition}

\theoremstyle{remark}
\newtheorem{remark}[thm]{Remark}

\numberwithin{equation}{section}

\renewcommand{\P}{\mathbb P}
\newcommand{\E}{\mathbb E}

\newcommand{\Z}{\mathbb Z}

\newcommand{\bbH}{\mathbb H}

\newcommand{\p}{\mathbb{P}}

\newcommand{\eps}{\varepsilon}

\newcommand{\bra}[1]{\left(#1\right)}
\newcommand{\sqbra}[1]{\left[#1\right]}

\newcommand{\den}[1]{\left\lVert#1\right\rVert}

\newcommand{\abs}[1]{\left\lvert#1\right\rvert}

\usepackage[margin=1in]{geometry}
\usepackage{tikz}
\usepackage{pgfplots}
\usepackage{extarrows}
\usepackage{titling}
\usepackage{commath}
\usepackage{wasysym}
\usepackage{mathtools}
\usepackage{titlesec}
\usepackage{bbm}
\usepackage{setspace}
\usepackage{enumitem}
\usepackage{booktabs}

\usepackage[textsize=tiny]{todonotes}

\usepackage{cite}

\def\P{\mathbb{P}}
\DeclareMathSymbol{\leqslant}{\mathalpha}{AMSa}{"36} % nicer `smaller or equal'
\DeclareMathSymbol{\geqslant}{\mathalpha}{AMSa}{"3E} % nicer `larger or equal'
\DeclareMathSymbol{\eset}{\mathalpha}{AMSb}{"3F}     % nicer `emptyset'
\renewcommand{\epsilon}{\varepsilon}

\makeatletter
\pgfdeclareshape{slash underlined}
{
  \inheritsavedanchors[from=rectangle] % this is nearly a circle
  \inheritanchorborder[from=rectangle]
  \inheritanchor[from=rectangle]{north}
  \inheritanchor[from=rectangle]{north west}
  \inheritanchor[from=rectangle]{north east}
  \inheritanchor[from=rectangle]{center}
  \inheritanchor[from=rectangle]{west}
  \inheritanchor[from=rectangle]{east}
  \inheritanchor[from=rectangle]{mid}
  \inheritanchor[from=rectangle]{mid west}
  \inheritanchor[from=rectangle]{mid east}
  \inheritanchor[from=rectangle]{base}
  \inheritanchor[from=rectangle]{base west}
  \inheritanchor[from=rectangle]{base east}
  \inheritanchor[from=rectangle]{south}
  \inheritanchor[from=rectangle]{south west}
  \inheritanchor[from=rectangle]{south east}
  \inheritanchorborder[from=rectangle]
  \foregroundpath{
    \southwest \pgf@xa=\pgf@x \pgf@ya=\pgf@y
    \northeast \pgf@xb=\pgf@x \pgf@yb=\pgf@y
    \pgf@xc=\pgf@xa
    \advance\pgf@xc by .5\pgf@xb
    \pgf@yc=\pgf@ya
    \advance\pgf@xc by -1.3pt
    \advance\pgf@yc by -1.8pt
    \pgfpathmoveto{\pgfqpoint{\pgf@xc}{\pgf@yc}}
    \advance\pgf@xc by  2.6pt
    \advance\pgf@yc by  3.6pt
    \pgfpathlineto{\pgfqpoint{\pgf@xc}{\pgf@yc}}
    \pgfpathmoveto{\pgfqpoint{\pgf@xa}{\pgf@ya}}
    \pgfpathlineto{\pgfqpoint{\pgf@xb}{\pgf@ya}}
 }
}
\tikzset{nomorepostaction/.code=\let\tikz@postactions\pgfutil@empty}

\makeatother

\pretitle{\begin{flushleft}\Large}
\posttitle{\par\end{flushleft}\vskip 0.5em
}
\preauthor{\begin{flushleft}}
\postauthor{
\\
\vspace{0.5em}
\footnotesize{$*$ The Division of Physics, Mathematics and Astronomy, California Institute of Technology}
\\
Email: \href{mailto:peaso@caltech.edu}{peaso@caltech.edu} and
\href{mailto:t.hutchcroft@caltech.edu}{t.hutchcroft@caltech.edu}\\
$\dagger$ Department of Mathematics and Statistics, McGill University\\
Email: \href{mailto:jana.kurrek@mail.mcgill.ca}{jana.kurrek@mail.mcgill.ca}

\end{flushleft}}
\predate{\begin{flushleft}}
\postdate{\par\end{flushleft}}
\title{\bf Double-exponential susceptibility growth in Dyson's hierarchical model with $|x-y|^{-2}$ interaction}
\renewenvironment{abstract}
 {\par\noindent\textbf{\abstractname.}\ \ignorespaces}
 {\par\medskip}

\titleformat{\section}
  {\normalfont\large \bf}{\thesection}{0.4em}{}

\author{{\bf Philip Easo$^*$, Tom Hutchcroft$^*$, and Jana Kurrek$^\dagger$}}

\begin{document}

\date{\small{\today}}

\maketitle

\setstretch{1.1}

\begin{abstract}
We study long-range percolation on the $d$-dimensional hierarchical lattice, in which each possible edge $\{x,y\}$ is included independently at random with inclusion probability $1-\exp ( -\beta \norm{ x-y }^{-d-\alpha} )$, where $\alpha>0$ is fixed and $\beta\geq 0$ is a parameter. This model is known to have a phase transition at some $\beta_c<\infty$ if and only if $\alpha<d$. We study the model in the regime $\alpha \geq d$, in which $\beta_c=\infty$, and prove that the
 % as $\beta$ tends to infinity 
susceptibility $\chi(\beta)$ (i.e.,  the expected volume of the cluster at the origin) satisfies
\[
    \chi(\beta) =
    \begin{cases}
        \beta^{\frac{d}{\alpha - d } - o(1)} \quad &\text{if $\alpha > d$,}\\
        e^{e^{ \Theta(\beta) }} \quad &\text{if $\alpha = d$}
    \end{cases}
     \qquad \text{ as $\beta \uparrow \infty$.}
\]
% When $\alpha = d$, we prove that $\chi(\beta)$ grows at least double-exponentially via a new \emph{sprinkled renormalisation} procedure.
 This resolves a problem raised by Georgakopoulos and Haslegrave (2020), who showed that $\chi(\beta)$ grows between exponentially and double-exponentially when $\alpha=d$. Our results imply that analogous results hold for a number of related models including Dyson's hierarchical Ising model, for which the double-exponential susceptibility growth we establish appears to be a new phenomenon even at the heuristic level.
\end{abstract}

\section{Introduction}

% In this paper we study critical long-range percolation on the hierarchical lattice. The hierarchical
% lattice Hd
% L
% is in some ways similar to the usual Euclidean lattice Z
% d but has additional symmetries
% and an exact recursive nesting structure that often makes hierarchical models of statistical mechanics much easier to understand than their Euclidean counterparts.

\emph{Hierarchical models} are toy models of statistical mechanics that exhibit similar phenomena to their Euclidean counterparts but which are much easier to study thanks to their exact recursive nesting structure. 
 % of statistical mechanics 
First introduced by Dyson \cite{dyson1969existence} in 1969, there is now a huge literature on hierarchical models within mathematical and theoretical physics, with Dyson's original paper having over 1000 citations; we refer the reader to \cite{dragovich2017p,dragovich2009p} for broad overviews of the use of hierarchical models in physics and \cite{bleher1987critical,MR3969983} for surveys of the rigorous analysis of critical phenomena in hierarchical models. Beyond their use in physics, hierarchical models have also been used to study epidemic spread \cite{ouboter2016stochastic} and population dynamics \cite{sawyer1983isolation}, where they may arguably be more realistic than either Euclidean or mean-field models.

 In this paper we study the low-temperature behaviour of hierarchical models \emph{at and below their lower-critical dimensions}, where phase transitions do not occur, a subject that has received relatively little prior treatment in the literature. We will see that the model displays particularly interesting behaviour at the lower-critical dimension itself, where it enjoys certain exact self-similarity properties.
  We focus on hierarchical \emph{percolation}, with our results immediately implying analogous results for various other models including the Ising and Potts models by standard stochastic domination properties.

%  to
% study the Ising model in 1969, there is now an extensive literature studying statistical mechanics on
% hierarchical lattices,  Hierarchical models have been particularly
% popular when e.g. analyzing spin systems via the renormalization group, where the exact recursive
% structure is extremely helpful [5, 20]. We refer the reader to e.g. [5, 38] for further background on
% hierarchical models, and to [5, Section 4.2] in particular for a detailed overview of the literature.

% The hierarchical lattice has an exact recursive structure that allows it to serve as a simplified version of the Euclidean lattice $\Z^d$. One key advantage is that it is more amenable to (rigorous) renormalisation arguments. In fact, the model was first introduced in the statistical physics literature by Dyson as an example where solutions to formulas for the Ising model could be explicitly obtained.

% \medskip

Let us now define the model.
 Given a dimension $d \in \mathbb N$ and a side-length $L \in \mathbb N$ with $L \geq 2$, the hierarchical lattice $\mathbb H_L^d$ is the group $\bigoplus_{i=1}^{\infty} (\mathbb Z / L \mathbb Z)^d$ equipped with the ultrametric given by $\norm{x-y} := L^{\max \{i :\;  x_i \not= y_i \} }$ for all distinct $x,y \in \mathbb H_L^d$. (This metric is \emph{not} a norm, but we use this notation to emphasise its analogy with the metrics on $\mathbb Z^d$ induced by norms on $\mathbb R^d$.) 
 % The model was first introduced in the statistical physics literature by Dyson \cite{dyson1969existence}, but it has also been used to study epidemic spread \cite{ouboter2016stochastic} and population dynamics \cite{sawyer1983isolation}.
The ultrametric balls of radius $L^n\mathbbm{1}(n>0)$ in this space are referred to as \textbf{$n$-blocks}, with the $n$-block containing the origin denoted by $\Lambda_n$. As a metric space, $\bbH^d_L$ can also be constructed recursively by taking
$\Lambda_0 = \{0\}$ and, for each $n\geq 0$, taking $\Lambda_{n+1}$ to be the union of $L^d$  disjoint copies of $\Lambda_n$ with distances defined by
$\|x-y\| = L^{n+1}$
 for each pair $x, y \in \Lambda_{n+1}$ belonging to distinct copies of $\Lambda_n$.
Given parameters $\alpha, \beta > 0$, we form a random graph $\omega$ with vertex set $\mathbb H_L^d$ by independently including each possible edge $xy:=\{x,y \}$ with probability $1-\exp ( -\beta \norm{ x-y }^{-d-\alpha} )$. We call this model \textbf{long-range percolation on the hierarchical lattice}. We denote its law by $\mathbb P_{\beta}$, omitting $\alpha$ because we typically think of it  as being fixed while $\beta$ varies. %More generally, given a subset $S \subseteq \mathbb H_L^d$, we write $\mathbb P_{\beta}^{S}$ for the model where we only allow edges whose endpoints both belong to $S$.

% \medskip

We are primarily interested in the geometry of the connected components of the random graph $\omega$, called \emph{clusters}. We write $K_x=K_x(\omega)$ for the cluster containing the element $x$, $x \leftrightarrow y$ to mean that $K_x = K_y$, and $x \leftrightarrow \infty$ to mean that $K_x$ is infinite. (Note that all these notions depend on the random graph $\omega$, but we suppress this from our notation when doing so does not cause confusion.) It is known that the \textbf{critical parameter} 
$\beta_c := \sup \{ \beta : \mathbb P_{\beta} ( o \leftrightarrow \infty ) = 0 \}$ is finite if and only if $d>\alpha$ \cite{dyson1969existence,dawson2013percolation,koval2012long}, so that $d=\alpha$ may be thought of as the \emph{lower-critical dimension} of the model. Since many of the most interesting questions about the model concern its behaviour at and near $\beta=\beta_c$, previous works have naturally focused on the case $0<\alpha<d$, where there is now a fairly good understanding of the model's critical behaviour \cite{hutchcroft2021critical,hutchcroft2022critical,koval2012long}.
% Let $o \in \mathbb H_L^d$ be the element with all coordinates zero. (Since the group of isometries of $\mathbb H_L^d$ acts transitively on the underlying set of group elements, our choice of which element to call $o$ does not matter.)
 % The most basic question is whether there is a non-trivial percolation phase transition, i.e.\! whether the critical parameter $\beta_c := \sup \{ \beta : \mathbb P_{\beta} ( 0 \leftrightarrow \infty ) = 0 \}$ is finite.
  % This was addressed by \cite{koval2012long,dawson2013percolation,georgakopoulos2020percolation}: $\beta_c < \infty$ if and only if $\alpha < d$.

  % \medskip

In this paper we instead study
% We consider
 the case $\alpha \geq d$, in which $\beta_c=\infty$. Although the model does not have a phase transition in this regime, the dependence of the model on the parameter $\beta$ remains very interesting. This is particularly true in the marginal case $\alpha=d$, where the model enjoys a certain exact self-similarity property as explained in \cref{sec:renormalization}.  To study this dependence on $\beta$, we focus in particular on the rate of divergence of the \textbf{susceptibility} $\chi(\beta):= \mathbb E_{\beta} \abs{K_0}$ 
  of the model, i.e.\ the expected 
   size of the cluster of the origin. 
   % (Equivalent formulations of our results in terms of the \emph{correlation length} are given below.)
    The susceptibility $\chi(\beta)$ is finite if and only if $\beta<\beta_c$ by sharpness of the phase transition \cite{MR894398,1901.10363,duminil2015new}, so that $\chi(\beta)$ blows up for finite values of $\beta$ if and only if $\alpha<d$. As such, it is plausible that the marginal case $\alpha=d$, where the model ``almost'' has a phase transition, might be characterized by $\chi(\beta)$ growing much faster as $\beta\to \infty$ when $\alpha=d$ than when $\alpha>d$. 
Indeed, the rapid growth of the susceptibility in the case $\alpha=d$ 
% The $\alpha=d$ case of \cref{thm:main} builds on the work of 
was previously studied by
 Georgakopoulos and Haslegrave\footnote{Interestingly, these authors had their own motivations to study a model equivalent to hierarchical percolation with $d=\alpha$, and were not aware of the previous literature on hierarchical models in physics.} \cite{georgakopoulos2020percolation}, who
  % (Proposition 8)
   proved that $e^{\Omega(\beta)} \leq \chi(\beta) \leq e^{e^{O(\beta)}}$ and suggested, based on numerical simulations, that the true growth might be of the form $e^{\Theta(\beta \log \beta)}$.

% \medskip

    Our main result states, surprisingly, that the susceptibility is in fact double-exponential in $\beta$ when $\alpha =d$, completely resolving \cite[Problem 8.1]{georgakopoulos2020percolation}. We also show that it grows as a power of $\beta$ when $\alpha>d$, so that there is indeed a striking quantitative distinction between the two cases.
    % this is true in a surprisingly strong sense: The susceptibility grows as a power of $\beta$ when $\alpha>d$ and double-exponentially in $\beta$ when $\alpha=d$. 
 % Since $\mathbb P_{\beta}\left(o \leftrightarrow \infty\right) = 0$ for all $\beta$, we instead turn to the study of the \emph{susceptibility} $\chi(\beta) := \mathbb E_{\beta} \abs{K_o}$. 
 % It follows from \cite{koval2012long,dawson2013percolation,georgakopoulos2020percolation} that $\chi(\beta) < \infty$ for every $\beta > 0$, and it follows from trivial considerations that $\chi(\beta) \to \infty$ as $\beta \to \infty$. 
 % The purpose of this paper is to understand the rate at which $\chi(\beta)$ tends to infinity. This will be qualitatively different in the $\alpha < d$ and $\alpha = d$ regimes. If $\alpha < d$, we find that the growth is polynomial with exponent $\frac{d}{\alpha - d}$. To prove this, we recursively (crudely) bound the volume of the largest cluster at a given scale in terms of the volumes of clusters at smaller scales. %The proof of the lower bound is a modification of a boostrap argument that the second author introduced in [ref] to study the $\alpha < d$ regime.

\begin{thm} \label{thm:main}
Let $d \geq 1$ and $L \geq 2$ be integers, let $\alpha \geq d$, and consider long-range percolation on the hierarchical lattice $\bbH^d_L$ in which each two vertices are connected by an edge with probability $1-\exp(-\beta\|x-y\|^{-d-\alpha})$. Then
\[
    \chi(\beta) =
    \begin{cases}
        e^{e^{ \Theta(\beta) }} \quad &\text{if $\alpha = d$}\\
                \beta^{\frac{d}{\alpha - d } - o(1)} \quad &\text{if $\alpha > d$}
    \end{cases}
     % \qquad \text{ as $\beta \uparrow \infty$.}
\]
as $\beta \to \infty$.
% $\chi(\beta) = \beta^{\frac{d}{\alpha - d } - o(1)}$ as $\beta \to \infty$.
\end{thm}

% \begin{thm} \label{thm:alpha>d}
% Let $d \geq 1$ and $L \geq 2$ be integers and let $\alpha \geq d$. Then
% \[
%     \chi(\beta) =
%     \begin{cases}
%         \beta^{\frac{d}{\alpha - d } - o(1)} \quad &\text{if $\alpha > d$,}\\
%         e^{e^{ \Theta(\beta) }} \quad &\text{if $\alpha = d$}
%     \end{cases}
%      \qquad \text{ as $\beta \uparrow \infty$.}
% \]
% % $\chi(\beta) = \beta^{\frac{d}{\alpha - d } - o(1)}$ as $\beta \to \infty$.
% \end{thm}

% The symbol $\asymp$ denotes an equality holding to within positive multiplicative constants.

% \medskip

% As is apparent from the exponent $\frac{d}{\alpha-d}$, our proof of the \cref{thm:alpha>d} above theorem breaks down when $\alpha = d$.
% The growth of the susceptibility in the case $\alpha=d$ 
% % The $\alpha=d$ case of \cref{thm:main} builds on the work of 
% had previously been studied by
%  Georgakopoulos and Haslegrave \cite{georgakopoulos2020percolation}, who
%   % (Proposition 8)
%    proved that there exist constants $c >0$ and $C < \infty$ such that $e^{c\beta} < \chi(\beta) \leq e^{e^{C\beta}}$ for every $\beta$ and suggested, based on numerical simulations, that the true growth might be of the form $e^{c n \log n}$. Surprisingly, and in contrast to this suggestion, \cref{thm:main} shows that their double-exponential upper bound is sharp; this completely resolves \cite[Problem 8.1]{georgakopoulos2020percolation}. 

\begin{remark}
The same self-similarity property that makes the $\alpha=d$ case particularly interesting from our perspective also leads the model's Euclidean ($\Z^d$) counterpart to have a rich and fractal-like large-scale geometry in the supercritical regime, with fractal dimension depending on the parameter~$\beta$ \cite{ding2013distances,baumler2022distances,baumler2022behavior}. The Euclidean model with $\alpha=d=1$ is also very interesting as an example of a percolation model undergoing a \emph{discontinuous} phase transition \cite{aizenman1988discontinuity,duminil2020long}, meaning that the close analogy between long-range percolation on the hierarchical and one-dimensional Euclidean lattices that holds for $\alpha<d$ \cite{MR4504407,baumler2022isoperimetric} breaks down rather badly at the point $\alpha=d$.
\end{remark}

% \begin{thm} \label{thm:alpha=d}
% Let $d \geq 1$ and $L \geq 2$ be integers and let $\alpha = d$. Then $\chi(\beta) =  e^{e^{ \Theta(\beta) }}$ as $\beta \to \infty$.
% \end{thm}

\medskip

\noindent \textbf{Consequences for other models.}
\cref{thm:main} immediately implies that analogous estimates hold for a large number of related models that are stochastically dominated above and below by Bernoulli percolation of appropriate parameters.
For example, the random cluster model on $\bbH^d_L$ with parameter $q\geq 1$, which in finite volume is defined by weighting the law of the Bernoulli percolation model we consider by a factor proportional to $q^{\#\text{clusters}}$, is always stochastically dominated by Bernoulli-$\beta$ percolation and stochastically dominates Bernoulli-$(\beta/q)$ percolation. It follows in particular that if $\chi(q,\beta)$ is the susceptibility (i.e., the expected size of the cluster of the origin) of the model with $\alpha=d$, then there exist positive constants $c$, $C$, and $\beta_0$ such that
\begin{equation}
e^{e^{\frac{c}{q} \beta}} \leq \chi(q,\beta) \leq e^{e^{C \beta}}
\end{equation}
for every $q\geq 1$ and $\beta\geq \beta_0$. (Note that for $\alpha\geq d$ the susceptibility can be defined without reference to boundary conditions since there is no phase transition and the Gibbs measure is always unique.) Using the Edwards-Sokal \cite{edwards1988generalization} coupling between the random cluster model and the Potts model when $q\geq 2$ is an integer, which identifies the susceptibilities of the two models, it follows that the same susceptibility estimates hold for the hierarchical Potts with interaction $J(x,y)=\|x-y\|^{-d-\alpha}$ for $\alpha =d$, and in particular to 
% 
 % This includes the random-cluster and Potts models with $q \geq 1$ on $\mathbb H_L^d$,
  % and in particular,
   Dyson's  hierarchical Ising model \cite{dyson1969existence} on $\bbH^1_2$ with interaction $|x-y|^{-2}$. Detailed background on these models and their relation to percolation can be found in \cite{MR2243761}. This striking double-exponential growth does not appear to have been discovered previously in any of these  models, even at a heuristic level.

\medskip

% \noindent \textbf{The correlation length.}

% \[
% \phi_\beta(\Lambda_n):=\sum_{x\in \Lambda_n} \sum_{y\in \bbH^d_L \setminus \Lambda_n} \left(1-e^{-\beta \|x-y\|^{-d-\alpha}}\right)\P_\beta\left(0 \leftrightarrow x \text{ in $\Lambda_n$}\right) 
% \]
% \[
% \beta_n = \sup \left\{\beta \geq 0 : \phi_\beta(\Lambda_n) \leq \frac{1}{2}\right\} \qquad \text{and} \qquad \beta_n^* = \max_{0\leq m \leq n} \beta_m
% \]
% \[
% \xi(\beta)=L^{n(\beta)} \qquad \text{where} \qquad n(\beta)=\inf\{n\geq 0: \beta \leq \beta^*_n\}.
% \]
% \[
% \sum_{x\in \Lambda_n} \P_\beta\left(0 \leftrightarrow x \text{ in $\Lambda_n$}\right)  \leq \chi(\beta) \leq 2 \sum_{x\in \Lambda_n} \P_\beta\left(0 \leftrightarrow x \text{ in $\Lambda_n$}\right) 
% \]
% \medskip

\noindent \textbf{About the proof.}  The proofs of the two cases $\alpha=d$ and $\alpha>d$ are very different, with the case $\alpha=d$ being much more delicate due to the model's resulting special self-similarity properties. The remainder of the paper is summarized as follows:
\begin{itemize}
\item In \cref{sec:renormalization} we introduce the renormalization framework that we use and give a very simple proof of the upper bound $\chi(\beta)=O(\beta^{d/(\alpha-d)})$ in the case $\alpha>d$. For the case $\alpha=d$, the most important idea introduced in this section is that by working with a certain mixed site-and-bond model, we can control the behaviour of percolation on large scales in terms of percolation on smaller scales, but with a change of parameters that depends on the size of the largest clusters on the smaller scale.
% .
 % This case is more delicate, with the model exhibiting a special self-similarity.

% If $\alpha > d$, we find that the growth is polynomial with exponent $\frac{d}{\alpha - d}$. To prove this, we recursively (crudely) bound the volume of the largest cluster at a given scale in terms of the volumes of clusters at smaller scales. 

\item In \cref{sec:alpha>d_lower} we complete the proof of the $\alpha>d$ case of \cref{thm:main} by proving an appropriate lower bound on $\chi(\beta)$ in this case. The proof of the lower bound is based on a modification of an induction-on-scales argument that the second author introduced in \cite{hutchcroft2021critical} to study the $\alpha < d$ regime; a more quantitative implementation of this argument is required to get a non-vacuous output in the case $\alpha>d$.

 \item In \cref{sec:alpha=d} we prove the $\alpha=d$ case of \cref{thm:main}. The proof of the lower bound, which is the primary new contribution of our paper, relies on a technique we call \emph{sprinkled renormalisation}: We use the renormalization technology introduced in \cref{sec:renormalization} to do an induction-on-scales in which we slightly increase the parameter $\beta$ each time we renormalize, taking care to not do this so many times that we increase $\beta$ to more than twice its original value. One interesting feature of this proof is that we \emph{double} the scale at each step of the induction, so that the side length of the block we consider grows doubly-exponentially in the number of steps taken; this turns out to make things work particularly nicely thanks to the self-similarity of the model. Finally, to keep the paper self-contained, in \cref{sec:alpha=d_upper} we give a new proof of the double-exponential upper bound of \cite{georgakopoulos2020percolation} based on the notion of \emph{correlation length} for hierarchical models introduced in \cite{MR4462652}.
 % We define a coarse-graining operation that allows us to relate the model at a given scale with parameter $\beta$ to the model at a lower scale with parameter $\tilde{\beta} < \beta$.
  % In order to compensate for this loss, the idea is to increase the original $\beta$ as we move to higher scales.
\end{itemize}

% \section{The case $\alpha>d$}
\section{The basic renormalisation framework}
\label{sec:renormalization}
In this section, we develop notation to describe how to control the percolation process at a given scale by the process at a smaller scale with a different effective parameter. Along the way we will deduce the upper bounds of \cref{thm:main} in the case $\alpha>d$.
 % using a very simple renormalization argument.

\medskip

\textbf{Blocks and their edges.}
Let $d \geq 1$ and $L \geq 2$ be integers, and let $\alpha > 0$ be a constant. For each integer $n \geq 0$ we refer to the ultrametric balls of radius $L^n \mathbbm{1}(n>0)$ in $\bbH^d_L$ as \textbf{$n$-blocks}. For each $x\in \bbH^d_L$ and $n\geq 0$ we write $\Lambda_n(x)$ for the $n$-block containing $x$ and write $\Lambda_n=\Lambda_n(0)$ for the $n$-block containing the origin. In other words, $\Lambda_n$ is the subset of $\mathbb H_L^d$ 
consisting of those $x$ with $x_i = 0$ for all $i > n$. We write $E_n(x)$ for the set of unordered pairs of distinct elements of $\Lambda_n(x)$, write $E=\bigcup_{n\geq 1} E_n$ for the set of unordered pairs of distinct elements of $\bbH^d_L$, and write $E_n=E_n(0)$.
 % More generally, for each $x\in \bbH^d_L$ we write $\Lambda_n(x)$ for the $n$-block containing $x$, write $E_n(x)$ for the set of unordered pairs of distinct elements of $\Lambda_n(x)$, and
 We also write $F_k=\bigcup_{x\in \bbH^d_L}E_k(x)$ for the set of all unordered pairs of distinct elements with distance at most $L^k$.
% 
% 
 % and let $E_n := \{ xy \subseteq \Lambda_n : x \not=y \}$. More generally, we will write $z + E_k$ to mean $\{ xy \subseteq z + \Lambda_k : x \not= y \}$. Set $E := \bigcup_{n \geq 1} E_n$. 
% 
 % For each integer $n \geq 0$ and configuration $\omega \in \{0,1\}^{E}$, let $M_n$ be an $\omega \cap E_n$-cluster of $\Lambda_n$ of maximum volume, i.e.\! an $\omega \cap E_n$-connected subset of $\Lambda_n$ that contains the maximum number of vertices among all such subsets.
 \medskip

\textbf{Block renormalisation.}
 We define $\pi : \mathbb H_L^d \to \mathbb H_L^d$ to be the \emph{left-shift} map defined by $(x_1,x_2,\dots) \mapsto (x_2,x_3,\dots)$ and define $\Phi : \{ 0,1 \}^E \to \{ 0,1\}^E$ mapping $\omega \mapsto \Phi[\omega]$ by setting $\Phi[\omega](e) =1$ if and only if there exists $xy \in E$ such that $\pi(x)\pi(y) = e$ and $\omega_{xy}=1$. This corresponds to zooming out by one scale, treating each copy of $\Lambda_1$ as a single vertex. 
In particular, if $x,y\in \bbH^d_L$ are connected in a configuration $\omega \in \{0,1\}^E$ then we must also have that $\pi(x)$ and $\pi(y)$ are connected in $\Phi[\omega]$, since any open path connecting $x$ and $y$ is mapped to an open path connecting $\pi(x)$ to $\pi(y)$. (The converse does not always hold.)
 \cref{lem:zoom} states that the effect of the map $\Phi$ on long-range percolation is simply to adjust the parameter $\beta$. Notice that when $\alpha = d$ the model is self-similar in the sense that $\beta$ remains unchanged. 

\begin{lem} \label{lem:zoom}
    For all $\beta > 0$, the pushfoward $\Phi_* \p_\beta$ is given by $\Phi_* \p_\beta = \p_{L^{d-\alpha} \beta}$. That is, if $\omega$ has law $\P_\beta$ then $\Phi[\omega]$ has law  $\P_{L^{d-\alpha}\beta}$.
\end{lem}

\begin{proof}
Independence is immediate, so it suffices to check that $\p_{L^{d-\alpha} \beta}$ has the correct marginals. Let $xy \in E$ be arbitrary. There are $L^{2d}$ edges $x'y' \in E$ with $\pi(x')\pi(y') = xy$, and each has $\den{x'-y'} = L \den{x-y}$. The probability that $\omega_{xy} = 0$ under $\Phi_* \p_\beta$ is the probability that $\omega_{x'y'} = 0$ for every one of these edges $x'y'$ under $\p_\beta$, so that
\[
    \Phi_* \p_\beta \bra{ \omega_{xy} = 0 } = \sqbra{e^{-\beta \bra{L \den{ x-y } }^{ - d - \alpha} }}^{L^{2d}} = \p_{ L^{d-\alpha}\beta } \bra{ \omega_{xy} = 0 }
\]
as required.
\end{proof}

This observation already lets us prove the upper bound of \cref{thm:main} in the case $\alpha>d$.

\begin{lem} \label{lem:alpha>d_ub}
If $\alpha>d$ then there exists a constant $C=C(d,L,\alpha)< \infty$ such that
 % for all $\beta \geq 1$,
    $\chi(\beta) \leq C \beta^{ \frac{d}{\alpha -d} }$ for every $\beta \geq 1$.
\end{lem}

\begin{proof}
 Recall that $K_0(\omega)$ denotes the cluster of the origin in the configuration $\omega\in \{0,1\}^E$. For each configuration $\omega \in \{ 0,1\}^E$ we have that
\[
|K_0(\omega)| = |\{x\in \bbH^d_L: 0 \leftrightarrow x\}| \leq |\{x\in \bbH^d_L: \pi(0) \leftrightarrow \pi(x)\}| = L^d |K_0(\Phi[\omega])|,
\]
% $\abs{K_o(\omega)} \leq L^d \abs{K_o(\Phi[\omega])}$,
 where both sides may be infinite. Taking expectations and applying \cref{lem:zoom}, we deduce that 
 \[\chi(\beta) \leq L^d \cdot \chi\bigl(L^{-(\alpha-d)} \beta\bigr)\] for every $\beta > 0$ and hence by induction that
\[\chi(\beta) \leq L^{dn} \cdot \chi\bigl( L^{ -(\alpha-d)n } \beta \bigr)\] for every $\beta>0$ and $n \geq 0$. 
% Take $n$ such that $L^{-(\alpha-d)n} \beta = 1$. For clarity we will assume that $n$ is a positive integer. The general case follows easily by appropriate \emph{floor-and-ceil} management. With this value of $n$, we obtain
Taking $n=\lceil \frac{1}{\alpha-d}\log_L \beta\rceil $ to be minimal such that $L^{-(d-\alpha) n}\beta \leq 1$ and using that $\chi(\beta)$ is an increasing function of $\beta$, we deduce that
\[
    \chi(\beta) \leq L^{d \lceil \frac{1}{\alpha-d}\log_L \beta \rceil } \cdot \chi(1) \leq L^d \chi(1) \beta^{\frac{d}{\alpha-d}}
\] 
for every $\beta \geq 1$. The claim follows with $C=L^d \chi(1)$ since $\chi(\beta) < \infty$ for every $\beta<\beta_c=\infty$.
\end{proof}

We next discuss a variation on this renormalization procedure that can be used to prove lower bounds.

\medskip

\textbf{Renormalisation with a mixed site-bond model.}
If we zoom out by $k$ scales by iterating the map $\Phi$ for $k$ steps,
% using 
% $\Phi^{k} := \overbrace{\Phi \circ \cdots \circ \Phi}^{k \text{ times}}$
% , then 
we lose all information about the configuration of edges
 in $F_k := \bigcup_{x\in \bbH^d_L}E_k(x)$, which join vertices at distance at most $L^k$. As we saw in \cref{lem:alpha>d_ub}, this is not necessarily a problem when proving upper bounds on our original model, where it may suffice to consider worst-case estimates in which every edge of $F_k$ is open.
 % by the the same model after opening every edge in $F_k$.
  To establish non-trivial lower bounds, however, we will require more information about the state of the edges in $F_k$.
Rather than keep track of \emph{all} relevant information about these small-scale edges, we will instead define an appropriate notion of what it means for a block to be `good', and keep track only of which blocks are good when re-scaling. 
   % The simplest such information is how many cosets $z + \Lambda_n$ in a ball contain an $\omega \cap (z + E_k)$-cluster that exceeds a particular volume.
   Since the goodness of different $k$-blocks will be independent of each other and independent of the status of edges not belonging to $F_k$, this naturally leads us to consider a \emph{mixed} site-bond percolation model. 

   For each $p \in [0,1]$, let $\mathbb Q_p$ be the law of the random subset $\eta$ of $\mathbb H_L^d$ obtained by independently including each element with probability $p$. Given $p \in [0,1]$ and $\beta > 0$, let $\p_{\beta,p}$ be the law of a random subgraph of $(\mathbb H_L^d, E)$, encoded as an element of $\Omega=\{0,1\}^{\mathbb H_L^d} \times \{0,1\}^E$ formed as follows: independently sample $\eta \sim \mathbb Q_p$ and $\omega \sim \p_\beta$, and then take the graph with vertex set $\{x : \eta (x) = 1\}$ and edge set $\{ xy : \omega_{xy} = \eta_x = \eta_y =1 \}$. As usual, we will abuse notation to think of $\omega$ and $\eta$ equivalently as the sets $\{e:\omega(e)=1\}$ and $\{x:\eta(x)=1\}$. Given a set $A\subseteq \bbH^d_L$, we refer to the connected  components of the subgraph of this graph induced by $A\cap \eta$ as \textbf{$(\eta,\omega)$-clusters in $A$} and say that two vertices $x,y\in A$ are \textbf{$(\eta,\omega)$-connected} in $A$ if they are in the same $(\eta,\omega)$-cluster in $A$. That is, two points $x,y\in A$ are $(\eta,\omega)$-connected in $A$ if there exists a path connecting $x$ to  $y$ all of whose vertices belong to $A\cap \eta$ and all  of  whose edges belong to $\omega$.

We next introduce the notation for zooming out by $k$ scales while only retaining edges between \emph{large} clusters. 
% For every point $z$ and configuration $\omega \in \{0,1\}^{E}$, fix a total order $\theta$ on the $\omega \cap (z + E_k)$-clusters of $z + \Lambda_k$. We do this to specify which large cluster to take when there are multiple. Let $\lambda \in [0,1]$ be a constant. A cluster is considered large if its volume exceeds $\lambda \abs{\Lambda_k}$. 
Fix an enumeration of $\bbH^d_L=\{x_1,x_2,\ldots\}$. Given $(\eta,\omega)\in \Omega$ and a finite set $A\subseteq \bbH^d_L$, we define
$K_\mathrm{max}(A)=K_\mathrm{max}(A;(\eta,\omega))$ to be an $(\eta,\omega)$ cluster in $A$ of maximal volume, where if there is more than one cluster of maximal volume we break ties using the fixed enumeration of $\bbH^d_L$ by taking the cluster containing a vertex of minimal label among the different maximal volume clusters. (By `volume' we just mean cardinality.) To lighten notation, we also write
\[
K^\mathrm{max}_n=K_\mathrm{max}(\Lambda_n;(\eta,\omega)) \qquad \text{ and } \qquad K^\mathrm{max}_n(z)=K_\mathrm{max}(\Lambda_n(z);(\eta,\omega))
\]
for each $n\geq 0$ and $z\in \bbH^d_L$ when the choice of $(\eta,\omega)$ is unambiguous.
% be the set of all elements $x \in \mathbb H_L^d$ such that 
% \begin{enumerate}
% \item $\Lambda_k(x)$ contains an $(\eta,\omega)$-cluster $K$ with $\abs{K} \geq \lambda \abs{\Lambda_k}$
% \item $x$ belongs to the minimal such cluster.
% \end{enumerate}
For each $\lambda>0$ and $k\geq1$ 
we define a map $\Psi^{\lambda,k} : \Omega \to \Omega$ by $\Psi^{\lambda,k}(\eta,\omega)=(\eta',\omega')$ where
\[
\eta'_x=\mathbbm{1}\Bigl(x=\pi^k(z) \text{ for some $z$ with $|K^\mathrm{max}_k(z)|\geq \lambda |\Lambda_k|$}\Bigr)
\]
and
\begin{multline*}
\omega'_{xy}=\mathbbm{1}\Bigl(x=\pi^k(z),y=\pi^k(w) \text{ for some $z,w$ with $\omega_{zw}=1$, }  \text{$z\in K^\mathrm{max}_k(z)$, and $w\in K^\mathrm{max}_k(w)$}\Bigr).
\end{multline*}
% sending $(\eta,\omega) \mapsto (\eta',\omega')$ as follows. Let $B$ be the set of all elements $z \in \mathbb H_L^d$ such that $\{x \in z + \Lambda_k : \eta_x=1\}$ contains an $\omega \cap (z + E_k)$-cluster $K$ with $\abs{K} \geq \lambda \abs{\Lambda_k}$ and $z$ belongs to the $\theta$-minimal such cluster. Set $\eta'_x :=1$ if and only if there exists $x' \in B$ with $\pi^{k} (x') = x$. Set $\omega_{xy} := 1$ if and only if there exists $x'y' \subseteq B$ with $\pi^{k} ( x'y' ) = xy$. 
This function has the following important property.

\begin{lem}
\label{lem:Psi_connectivity}
Let $(\eta,\omega) \in \Omega$, let $k\geq 1$ and let $\lambda>0$.
If $x,y\in \bbH^d_L$ and $n\geq 1$ are such that $x \in K^\mathrm{max}_k(x)$, $y \in K^\mathrm{max}_k(y)$, and $\pi^k(x)$ is $\Psi^{\lambda,k}(\eta,\omega)$-connected to $\pi^k(y)$ in $\Lambda_n$ then $x$ and $y$ are $(\eta,\omega)$-connected in $\Lambda_{n+k}$.
In particular, 
% letting $\Psi^{\lambda,k}(\eta,\omega)=(\eta',\omega')$ we have that
\[\left|K_\mathrm{max}\left(\Lambda_{n+k};(\eta,\omega)\right)\right| \geq \lambda L^{dk}\left|K_\mathrm{max}\left(\Lambda_{n};\Psi^{\lambda,k}(\eta,\omega)\right)\right| \]
for every $n\geq 1$.
\end{lem}

\cref{lem:zoom+} describes how the effect of $\Psi$ on a mixed percolation process $\p_{\beta,p}$ can be bounded by the effect of adjusting the parameters $\beta$ and $p$.  
% For a mixed configuration $(\eta,\omega)$, we write $M_k(\eta,\omega)$ for an $\omega \cap E_k$-cluster of $\{ x \in \Lambda_k : \eta_x =1 \}$ with maximum volume.
 % (Note that if $\eta_x=0$ for all $x \in \Lambda_k$, then $M_k(\eta,\omega) = \emptyset$ but $\abs{M_k(\omega)}=1$.) 

\begin{lem} \label{lem:zoom+}
For each $p,\lambda \in [0,1]$, $\beta > 0$, and $k\geq 0$ let
% \[
%     \Psi_*^{\lambda,k} \p_{\beta,p} \geq_{\mathrm{st}} \p_{\beta',p'},
% \]
\[p' := \p_{\beta,p}( |K^\mathrm{max}_k|\geq \lambda L^{dk}) \qquad \text{and} \qquad \beta' := \lambda^2 L^{k(d-\alpha)} \beta.\]
The law of the random graph with vertex set $\{x : \eta (x) = 1\}$ and edge set $\{ xy : \omega_{xy} = \eta_x = \eta_y =1 \}$ under the measure $\Psi_*^{\lambda,k} \p_{\beta,p}$ stochastically dominates the law of the same random graph under the measure $\p_{\beta',p'}$.
\end{lem}

% Here $\geq_{\mathrm{st}}$ denotes stochastic domination.
 Notice that when $\alpha = d$, the parameter $\beta$ is simply replaced by $\beta':=\lambda^2 \beta$.

\begin{proof}
Sample $(\eta,\omega) \sim \p_{\beta,p}$ and set $ (\eta',\omega') := \Psi^{\lambda,k}(\eta,\omega)$. Notice that $\eta'$ is determined by $\eta$ and $\omega \cap F_k$. By construction of $\Psi$, the definition of $p'$, and transitivity, we have that $\eta' \sim \mathbb Q_{p'}$. Hence, it suffices to check that if we fix realisations of $\eta$ and $\omega \cap F_k$ and independently sample $\omega \cap {\{xy \in E \backslash F_k : \eta_x = \eta_y =1 \}}$ according to its law under $\p_\beta$, then the law of $\omega'$ stochastically dominates the law of the restriction of a sample of $\p_{\beta'}$ to the set of edges $xy$ with $\eta'_x = \eta'_y = 1$. Indeed, notice that the state of the edges in $\omega'$ are independent of each other and that, arguing as in the proof of \cref{lem:zoom}, for every edge $xy$ with $\eta'_x = \eta'_y = 1$, the probability that $\omega_{xy} = 0$ is
\[
    \prod_{\substack{x'\in \bbH^d_L :\\ \pi^{k}(x') = x} } \prod_{\substack{y'\in \bbH^d_L :\\ \pi^{k}(y') = y}} \p_\beta \bra{ \omega_{x'y'} = 0 } \geq \sqbra{e^{ -\beta \bra{ L^{k} \den{x-y}  }^{-d-\alpha}}}^{ (\lambda L^{dk})^2 } = \p_{\beta'}(\omega_{xy} = 0)
     % \qedhere
\]
as required.
\end{proof}

\section{Lower bounds in the case $\alpha>d$}
\label{sec:alpha>d_lower}

% We give a proof by induction for the lower bound.
In this section we prove the lower bound of \cref{thm:main} in the case $\alpha>d$. The proof uses an induction on scales that adapts the ideas of \cite[Lemma 2.4]{hutchcroft2021critical} to the large $\alpha$ regime.
Compared to the treatment of \cite{hutchcroft2021critical}, our argument is both more quantitative (which is necessary to get a non-vacuous statement in the large $\alpha$ regime), and is made more streamlined by the use of the renormalization notation established in the previous subsection.
 % Adapting this statement to the setting $\alpha>d$ requires a more quantitative statement to remain meaningful.
 % We nevertheless include a full proof because the statement of that lemma is not strictly applicable in our context and because our renormalisation notation allows for a more streamlined presentation of the argument.

\medskip

  Our argument involves repeatedly zooming out by $k$ scales, where $k$ is a carefully chosen integer depending on $\beta$. More precisely, we pick $k=k(\beta) : = k_0 \vee \lfloor \sqrt{ \log \beta } \rfloor$ where $k_0\geq 1$ is an integer that is sufficiently large to guarantee that 
  % for every $\beta \geq 1$, $k$ satisfies
\begin{equation} \label{eq:technical1}
    \mathbb Q_{\frac{1}{2}} \Big( \abs{ \{x \in \Lambda_{k} : \eta_x = 1 \} } \geq \frac{1}{3}L^{dk}  \Big) \geq 1 - \frac{1}{4}
\end{equation}
for every $k\geq k_0$; such a constant $k_0$ exists by the weak law of large numbers. We will zoom out exactly $\ell=\ell(\beta)$ times where $\ell\geq 1$ is the largest integer such that
\begin{equation} \label{eq:l_condition}
   L^{2dk} \exp \bra{ - \frac{\beta}{9^{\ell} L^{ (\alpha-d)\ell k }  } \cdot \frac{1}{L^{(d+\alpha)k}} } \leq \frac{1}{4}. 
\end{equation}
If no such $\ell$ exists (which may be the case when $\beta$ is small) we set $\ell=0$.
% As $\beta\to\infty$ we have that
% \[
% k(\beta)\sim H \sqrt{\log \beta} \text{ and } \ell(\beta) 
% \]
% As in the proof of \cref{lem:alpha>d_ub}, we will assume for clarity that $k$ and $l$ are positive integers. Again the general case follows easily by appropriate floor-and-ceil  management.

\begin{lem} \label{lem:induction_alpha>d}
        $\mathbb{P}_\beta(\lvert K^\mathrm{max}_{rk}\rvert \geq 3^{-r}L^{drk}) \geq \frac{1}{2}$ for every $\beta \geq 1$ and $r \in \{ 0, 1 , \dots, \ell \}$.
\end{lem}

\begin{proof}
Fix $\beta \geq 1$. We proceed by induction on $r$. The result is trivial for $r = 0$. Assume that the result holds for some $r \in \{ 0, \dots, \ell-1 \}$. Letting $\mathbf{1}\in \{0,1\}^{\bbH^d_L}$ be the all-ones function, we have by \cref{lem:Psi_connectivity} that
% For each configuration $\omega \in \{ 0,1 \}^E$,
\[
    % | \geq \frac{\abs{\Lambda_{rk}}}{3^r} \cdot \abs{M_{k} ( \Psi^{\frac{1}{3^r} , rk} (\omega) ) }.
    \left|K_\mathrm{max}\left(\Lambda_{(r+1)k};(\mathbf{1},\omega)\right)\right| \geq 3^{-r} L^{drk}\left|K_\mathrm{max}\left(\Lambda_{k};\Psi^{3^{-r},rk}(\mathbf{1},\omega)\right)\right| 
\]
for each $\omega \in \{0,1\}^E$ and hence that
% It follows that
\[\mathbb{P}_{\beta, 1}\left(\bigl|K^\mathrm{max}_{(r+1)k}\bigr| \geq 3^{-r-1}L^{d(r+1)k}\right) \geq \Psi_*^{3^{-r}, r k} \mathbb{P}_{\beta, 1}\left(\left|K^\mathrm{max}_{k}\right| \geq \frac{1}{3}L^{dk}\right). \]
Applying \cref{lem:zoom+} with $p := 1$, we deduce that
\begin{equation} \label{eq:apply_gamma}
    \p_\beta \bra{ \left|K^\mathrm{max}_{(r+1)k}\right| \geq 3^{-r-1} L^{d(r+1)k} } \geq \p_{\beta',p'} \bra{ \left|K^\mathrm{max}_{k}\right| \geq \frac{1}{3}L^{dk} },
\end{equation}
where 
\[p' = \p_\beta \bra{ \left|K^\mathrm{max}_{(r+1)k}\right|  \geq 3^{-r}L^{drk} } \qquad \text{and} \qquad \beta' := \frac{\beta}{3^{2r} L ^{(\alpha-d)rk}}.\]
We have by the induction hypothesis that $p' \geq \frac{1}{2}$, so that
\begin{align}
\mathbb{P}_\beta(\lvert K^\mathrm{max}_{rk}\rvert \geq 3^{-r}L^{drk}) &\geq \p_{\beta',1/2} \bra{ \left|K^\mathrm{max}_{k}\right| \geq \frac{1}{3}L^{dk} }\nonumber\\
&\geq \mathbb Q_{1/2} \Big( \abs{ \{x \in \Lambda_{k} : \eta_x = 1 \} }\geq \frac{1}{3}L^{dk}  \Big)\cdot \p_{\beta'} \bra{ \omega_e = 1 \; \forall e \in E_k }.\label{eq:wasteful}
\end{align}
Our choice of $k$ ensures that $\mathbb Q_{1/2} \Big( \abs{ \{x \in \Lambda_{k} : \eta_x = 1 \} }\geq \frac{1}{3}L^{dk}  \Big) \geq \frac{3}{4}$, while we have by a union bound and our choice of $\ell$ that
% % 
% % . So by \cref{eq:technical1},
% \[
%     \mathbb Q_{p'} \Big( \abs{ \{x \in \Lambda_{k} : \eta_x = 1 \} } \geq \frac{\lvert \Lambda_{k} \rvert}{3}  \Big) \geq 1- \frac{1}{4}.
% \]
% By a union bound and \cref{eq:l_condition},
\begin{align*}
    \p_{\beta'} \bra{ \omega_e = 1 \; \forall e \in E_k } &\geq 1 - \abs{E_k} \max_{e \in E_k} \p_{\beta'} \bra{ \omega_e = 0 } 
    &\geq 1 - L^{2dk} \text{exp}{\bra{ - \frac{\beta}{3^{2r} L^{ (\alpha-d)rk }  } \cdot \frac{1}{L^{(d+\alpha)k}} }} \geq \frac{3}{4},
\end{align*}
so that
\[
\mathbb{P}_\beta\left(\lvert K^\mathrm{max}_{rk}\rvert \geq 3^{-r}L^{drk}\right) \geq \frac{9}{16}\geq \frac{1}{2}
\]
as claimed.
% So by another union bound,
% \[
    % \p_{\beta',p'} \bra{ \abs{M_k} \geq \frac{\abs{\Lambda_k}}{3} } \geq 1 - \frac{1}{4} - \frac{1}{4} = \frac{1}{2}.
% \]
% Plugging this into \cref{eq:apply_gamma} completes the proof of the inductive step and hence the proof of the lemma.
\end{proof}

The following proposition implies the lower bound of \cref{thm:main} in the case $\alpha>d$, and gives an explicit estimate on the $o(1)$ term appearing in that estimate.

\begin{prop} \label{prop:alpha>d_lb}
If $\alpha>d$ then there exists a constant $C=C(d,L,\alpha) < \infty$ such that
\[
    \chi(\beta) \geq \beta^{ \frac{d}{\alpha - d} - \frac{C}{ \sqrt{\log \beta}} }
\]
for every $\beta \geq 2$.
\end{prop}

\begin{proof}
Fix $\beta \geq 2$. We have by transitivity that
 % that $\p_\beta \bra{ \lvert M_{lk}\rvert \geq \frac{\abs{\Lambda_{lk}}}{3^l} } \geq \frac{1}{2}$. It follows by transitivity that
\[
  \chi(\beta) \geq \mathbb{E}_\beta\left[|K^\mathrm{max}_n| \mathbbm{1}(0\in |K^\mathrm{max}_n|) \right] = \frac{1}{|\Lambda_n|}\sum_{x\in \Lambda_n} \mathbb{E}_\beta\left[|K^\mathrm{max}_n| \mathbbm{1}(x\in |K^\mathrm{max}_n|) \right]  = L^{-dn}\mathbb{E}_\beta\left[|K^\mathrm{max}_n|^2 \right]
  % \p_\beta \bra{ \abs{K_o(\omega \cap E_{kl})} \geq \frac{ \abs{ \Lambda_{kl} } }{3^l} } \geq \frac{1}{2} \cdot \frac{1}{3^l},
\]
for every $n\geq 1$, and hence by \cref{lem:induction_alpha>d} that
\begin{equation} \label{eq:chi_raw_alpha>d}
    \chi(\beta) \geq \frac{1}{2} \cdot \frac{L^{dk\ell}}{9^\ell}.
\end{equation}
To complete the proof, we use the definitions of $k$ and $\ell$ to compute that
\[
9^\ell L^{(\alpha-d)k\ell} \sim \frac{1}{2d \log L} \frac{k\beta}{L^{(d+\alpha)k}}= \beta \cdot \exp\left[-O(\sqrt{\log \beta})\right] \qquad \text{ as $\beta \to \infty$},
\]
where $\sim$ means that the ratio of the two sides converges to $1$ in the relevant limit,
% as $\beta \to \infty$, 
so that
\[
\ell \sim \frac{1}{(\alpha -d) \log L} \frac{\log \beta}{k} \sim \frac{1}{(\alpha -d)\log L} \sqrt{\log \beta} \qquad \text{ as $\beta \to \infty$}
\]
and
\[
\frac{L^{dk\ell}}{9^\ell} = 9^{-\frac{\alpha}{\alpha-d} \ell}\left(9^\ell L^{(\alpha-d)k\ell}\right)^{\frac{d}{\alpha-d}} =\beta^{\frac{d}{\alpha-d}} \cdot \exp\left[-O(\sqrt{\log \beta})\right] \qquad \text{ as $\beta \to \infty$,}
\]
 where all implicit constants may depend on $d$, $\alpha$, and $L$. Substituting this estimate into \eqref{eq:chi_raw_alpha>d} implies the claim.
\end{proof}

\begin{remark}
It may seem that the estimate \eqref{eq:wasteful} is very wasteful: The Erd\H{o}s-R\'enyi random graph contains a giant cluster well before every edge is open, and it would suffice for the rest of the analysis to have $\beta' L^{-(d+\alpha)k} \gg L^{-dk}$ rather than $\beta' L^{-(d+\alpha)k} \gg k $ as we require. It turns out, however, that carrying the analysis through with this improvement (and with the resulting optimal choices of $k$ and $\ell$) merely leads to a better value of the constant $C$ in \cref{prop:alpha>d_lb}. 
% On the other hand, it seems likely that in fact $\chi(\beta)=\Theta(\beta^{d/(\alpha-d)})$ when $\alpha>d$.
\end{remark}

% \Cref{thm:alpha>d} follows immediately from \Cref{lem:alpha>d_ub} and \Cref{lem:alpha>d_lb}.

\section{The case $\alpha=d$}
\label{sec:alpha=d}

In this section we prove the $\alpha=d$ case of \cref{thm:main}. We begin with the lower bound, which is the primary new result of the paper, before giving a short self-contained treatment of the upper bound (which recovers the results of \cite{georgakopoulos2020percolation}) in \cref{sec:alpha=d_upper}. The arguments of \cref{sec:alpha=d_lower} rely on the renormalization framework developed in the previous sections while those of \cref{sec:alpha=d_upper} use a separate argument, which draws in part on the techniques of \cite[Section 4]{MR4462652}.

\subsection{Lower bounds}

\label{sec:alpha=d_lower}

% \section{The $\alpha > d$ case}

% In this section, we prove \cref{thm:alpha>d}. Let $d \geq 1$ and $L \geq 2$ be integers. Fix $\alpha > d$. We start by proving the upper bound, which is easier.

% \begin{lem} \label{lem:alpha>d_ub}
% There is a constant $C < \infty$ such that for all $\beta \geq 1$,
% \[
%     \chi(\beta) \leq C \beta^{ \frac{d}{\alpha -d} }.
% \]
% \end{lem}

% \begin{proof}
% Let $\beta \geq 1$. For every configuration $\omega \in \{ 0,1\}^E$, we have $\abs{K_o(\omega)} \leq L^d \abs{K_o(\Phi[\omega])}$, where both sides may be infinite. By taking expectations and applying \cref{lem:zoom}, we deduce that $\chi(\beta) \leq L^d \cdot \chi\bra{L^{-(\alpha-d)} \beta}$. So by induction, $\chi(\beta) \leq L^{dn} \cdot \chi( L^{ -(\alpha-d)n } \beta )$ for every $n \geq 0$. Take $n$ such that $L^{-(\alpha-d)n} \beta = 1$. For clarity we will assume that $n$ is a positive integer. The general case follows easily by appropriate \emph{floor-and-ceil} management. With this value of $n$, we obtain
% \[
%     \chi(\beta) \leq L^{dn} \cdot \chi(L^{ -(\alpha-d) n } \beta) = \beta^{\frac{d}{\alpha-d}} \cdot \chi(1).
% \] 
% Since $\chi(1) < \infty$ \cite{dawson2013percolation,koval2012long}, we are done with $C := \chi(1)$.
% \end{proof}

% \section{The $\alpha = d$ case}
In this section we prove the lower bound of \cref{thm:main} in the case $\alpha=d$.

\begin{prop} \label{prop:alpha=d_lb}
If $\alpha=d$ then there exists $c=c(d,L) > 0$ such that 
    $\chi(\beta) \geq e^{ e^{c \beta} }$
for every $\beta \geq 1$.
\end{prop}
 % Let $d \geq 1$ and $L\geq 2$ be integers. Fix $\alpha = d$. Georgakopoulos and Haslegrave \jana{[ref,thm]} proved that there are constants $c > 0$ and $C < \infty$ such that $e^{c \beta} \leq \chi(\beta) \leq e^{e^{C\beta} }$ for all $\beta \geq 1$. So to prove \cref{thm:alpha=d}, it suffices to prove that in fact there is a constant $c >0$ such that $\chi (\beta) \geq e^{e^{c\beta}}$ for all $\beta \geq 1$. 
 % The inductive step in our proof will consist in applying the following lemma.

 We will prove \cref{prop:alpha=d_lb} using a ``sprinkled renormalization'' argument, in which we slightly increase the parameter each time we zoom out. An interesting feature of the proof is that, rather than going up one scale at a time, we instead \emph{double} the scale at each induction step, so that the side-length of the block considered at the $i$th induction step is double-exponential in $i$. We will rely on two auxiliary lemmas, the first of which encapsulates the induction step.

\begin{lem}[Inductive estimate] \label{lem:alpha=d_induction}
% Let $\eps \in [0,\frac{1}{2}]$ be a constant and let $n \geq 1$ be an integer. For all parameters $p \in [0,1]$ and $\beta \geq 1$,
If $\alpha =d$ then the implication
\[
    \left(\p_{\beta,p} \left(|K^\mathrm{max}_n|\geq (1-\eps) \abs{\Lambda_n} \right) \geq p \right) \quad \implies \quad
    \left(\p_{(1+6\eps)\beta,p} \left(|K^\mathrm{max}_{2n}|\geq (1-2\eps) \abs{\Lambda_n} \right) \geq p \right)
     % \p_{(1+6\eps)\beta, p} \bra{ \abs{M_{2n}} \geq (1-2\eps) \abs{\Lambda_{2n}} } \geq p.
\]
holds for every $p\in [0,1]$, $\beta \geq 1$, $0<\eps\leq 1/2$, and $n\geq 1$.
\end{lem}

The next auxiliary lemma establishes the base case of the induction. This base case estimate is more delicate than one might expect, and we do \emph{not} take our base case to be $n=0$. Rather, for the induction to work, we need to find a base scale $n_0$ where the probability that $\abs{K^\mathrm{max}_n}$ is close to $\abs{\Lambda_n}=L^{dn}$ under $\p_{\beta,p}$ is at least $p$, where $p$ is a constant that is bounded away from zero. (NB: It is very important that the $p$ appearing as the parameter in $\p_{\beta,p}$ and the $p$ appearing as the lower bound on the probability of the relevant event are equal!) To address the increase in $\beta$ along the induction, we begin with a lower initial parameter $\frac{\beta}{2}$.

\begin{lem}[Base case] \label{lem:alpha=d_base}
If $\alpha=d$ then there exists a constant $\beta_* =\beta_*(d,L)< \infty$ such that if  we define
\[
\delta = \delta(\beta)=\exp\left[-L^{-9d}\beta\right] \qquad \text{ and } \qquad n_0=n_0(\beta)=  \left\lceil \frac{2 \beta }{L^{9d} d \log L} \right\rceil
\]
then
\[
    \p_{\frac{1}{2}\beta,1-\delta} \bra{ |K^\mathrm{max}_{n_0}| \geq ( 1 -2 \delta) L^{dn_0} } \geq 1 - \delta
\]
for every $\beta \geq \beta_*$.
% where $b := \frac{1}{L^{9d}}$ and $a:=a(\beta)$ is the smallest real number with $a \geq \frac{2}{L^{9d} d \log L}$ such that $a\beta$ is a positive integer.
\end{lem}

Before proving these lemmas, let us first see how they imply \cref{prop:alpha=d_lb}.

\begin{proof}[Proof of \cref{prop:alpha=d_lb}]
The idea is to repeatedly apply \cref{lem:alpha=d_induction} as many times as possible beginning with \cref{lem:alpha=d_base}.
 There are two constraints. First, the value of $\eps$ will eventually increase beyond $\frac{1}{2}$, at which point the hypothesis of \cref{lem:alpha=d_induction} will no longer be met. Second, our parameter, which starts at $\frac{1}{2}\beta$, will eventually increase beyond $\beta$, at which point we can no longer bound $\p_{\beta,p}$ with the current estimate. A satisfactory lower bound on the number of times that we can iterate will be $\ell := \ell(\beta)=\lceil -\log_2 (100 \delta) \rceil$, which satisfies
\[
   \frac{1}{50} \leq 2^{\ell} \delta \leq \frac{1}{100}
\]
for every $\beta \geq 1$.
We may assume that the constant $\beta_*$ is sufficiently large that $\delta(\beta)\leq 1/2$ and $\ell(\beta)\geq 1$ for every $\beta \geq \beta_*$.
% Again, for clarity we will assume that $l$ is a positive integer, with the general case following by appropriate floor-and-ceil management.
Fix $\beta \geq 2 \beta_*$, and for each $0\leq r \leq \ell(\beta)$ let 
\[ \delta_r = 2^r \delta , \qquad n_r = 2^r n, \qquad  \text{ and } \qquad \beta_r= \exp\left[ 12 \delta_r \right]\frac{\beta}{2}, \]
so that $\beta/2\leq \beta_r \leq e^{0.12}\beta/2\leq \beta$ for every $0\leq r \leq \ell$ by choice of $\ell$.
We claim that
% \begin{lem} \label{lem:alpha=d_repeated_application}
% For every $\beta \geq \beta_*$ and every $r \in \{ 0,\dots,l \}$,
\begin{equation}
\label{eq:induction_claim}
    \p_{ \beta_r, 1 - \delta } \bra{ |K^\mathrm{max}_{n_r}| \geq (1 - 2\delta_r) L^{dn_r} } \geq 1- \delta
\end{equation}
for every $0\leq r \leq \ell$.
% where $\lambda_r(\beta) := \frac{1}{2}e^{ 6 \cdot 2^{r+1} e^{-b\beta} } \beta$.
% \end{lem}
% \begin{proof}
We proceed by induction on $r$. When $r = 0$, the result follows from \cref{lem:alpha=d_base} since $\beta_0 \geq \beta/2$. Assume that the result holds for some $r \in \{ 0,\dots,l-1 \}$. 
% Define $\eps := 2^{r+1}e^{-b\beta}$. 
Since $r \leq \ell$, the definition of $\ell$ guarantees that $\delta_r \leq \frac{1}{2}$ and hence by \cref{lem:alpha=d_induction} (applied with $p=1-\delta$ and $\eps=2\delta_r$) that
\[
    \p_{  (1+12\delta_r )\beta_r , 1 - \delta } \bra{ |K^\mathrm{max}_{n_{r+1}}| \geq (1 - 4 \delta_{r}) L^{dn_{r+1}}  } \geq 1- \delta.
\]
We can therefore conclude the induction step by noting that $4 \delta_{r} = 2\delta_{r+1}$ and
 % follows from the observation that
\[
   (1+12\delta_r) \beta_r \ \leq e^{  12 \delta_{r} } \beta_r  =\exp[24 \delta_r] \frac{\beta}{2} = \beta_{r+1} .
    % \qedhere
\]
% \end{proof}

% Our lower bound on $\chi(\beta)$ now follows similarly to the proof of \cref{lem:alpha>d_lb}

% \begin{proof}
It remains to deduce the claimed lower bound on $\chi(\beta)$ from \eqref{eq:induction_claim}. As in the proof of \cref{prop:alpha>d_lb}, it follows from \eqref{eq:induction_claim} and transitivity that
\[
\chi(\beta) \geq L^{-dn_\ell} \E_\beta \left[|K^\mathrm{max}_{n_\ell}|^2\right] \geq (1-\delta)(1-2\delta_\ell)^2 L^{dn_\ell} \geq \frac{1}{8} L^{dn_\ell}
\]
for every $\beta \geq 2 \beta_*$. The claim follows since
\[
n_\ell =2^\ell n_0 =\Theta\left( e^{L^{-9d}\beta} \beta \right) = e^{\Theta(\beta)}
\]
as  $\beta \to \infty$ by definition of $\ell$ and $n_0$. \qedhere

% By reducing $c$, it suffices to prove the claim for all sufficiently large $\beta$. In particular, by taking $\beta \geq \beta_0$, we may apply \cref{lem:alpha=d_repeated_application}. By our choice of $l$, we have $2^{l+1}e^{-b\beta} \leq \frac{1}{2}$ and $\lambda_l(\beta) \leq \beta$.  By taking $\beta$ sufficiently large, we may assume that $1-e^{-b\beta} \geq \frac{1}{2}$. So by \cref{lem:alpha=d_repeated_application} with $r= l$, 
% \[
%     \p_{ \beta} \bra{ \lvert M_{2^{l}a\beta} \rvert \geq \frac{\lvert \Lambda_{2^l a \beta} \rvert}{2} } \geq \frac{1}{2}.
% \]
% By the same argument as in the proof of \cref{lem:alpha>d_lb}, since the isometry group of $\Lambda_{2^l a\beta}$ is transitive, it follows that
% \[
%     \chi(\beta) \geq \frac{1}{2} \cdot \frac{1}{2} \cdot \frac{1}{2} \lvert \Lambda_{2^la\beta}\rvert.
% \]
% The result follows by plugging in the definition of $l$.
% \end{proof}

% \cref{thm:alpha=d} now follows immediately from the lower bound given by \cref{lem:alpha=d_lb} and the upper bound from [ref,thm].
\end{proof}

We now prove the two auxiliary lemmas, \cref{lem:alpha=d_induction,lem:alpha=d_base}. We begin with the inductive estimate \cref{lem:alpha=d_induction}, which is a simple consequence of \cref{lem:Psi_connectivity,lem:zoom+}.

\begin{proof}[Proof of \cref{lem:alpha=d_induction}]
% In every configuration $\omega \in \{ 0,1 \}^E$,
We may apply \cref{lem:Psi_connectivity} as in the proof of \cref{lem:induction_alpha>d} to obtain that
% For each configuration $\omega \in \{ 0,1 \}^E$,
\[
    % | \geq \frac{\abs{\Lambda_{rk}}}{3^r} \cdot \abs{M_{k} ( \Psi^{\frac{1}{3^r} , rk} (\omega) ) }.
    \left|K_\mathrm{max}\left(\Lambda_{2n};(\eta,\omega)\right)\right| \geq (1-\eps) L^{dn} \left|K_\mathrm{max}\left(\Lambda_{n};\Psi^{1-\eps,n}(\eta,\omega)\right)\right| 
\]
for each $(\eta,\omega) \in \Omega$ and hence that
% It follows that
\[\mathbb{P}_{(1+6\eps)\beta, p}\left(\bigl|K^\mathrm{max}_{2n}\bigr| \geq (1-\eps)^2 L^{2dn}\right) \geq \Psi_*^{1-\eps, n} \mathbb{P}_{(1+6\eps)\beta, p}\left(\left|K^\mathrm{max}_{n}\right| \geq (1-\eps)L^{dn}\right). \]
% 
% \[
%     \abs{M_{2n}(\omega)} \geq (1-\eps) \cdot \abs{\Lambda_n} \cdot \abs{ M_n ( \Psi^{1-\eps,n} (\omega)  ) }.
% \]
% So if $\lvert M_n ( \Psi^{1-\eps,n} (\omega) ) \rvert \geq (1-\eps) \abs{\Lambda_n }$, then $\abs{ M_{2n} (\omega) } \geq (1-\eps)^2 \abs{\Lambda_n} \abs{\Lambda_n} \geq (1-2\eps) \abs{\Lambda_{2n}}$. In particular,
% \begin{equation} \label{eq:alpha=d_key_lemma_first_bound}
%     \p_{ (1+6\eps)\beta,p} \bra{ \abs{M_{2n}} \geq (1-2\eps) \abs{\Lambda_{2n}} } \geq \Psi_*^{1-\eps,n} \p_{(1+6\eps)\beta,p} \bra{ \abs{M_n} \geq (1-\eps) \abs{\Lambda_n} }.
% \end{equation}
Now, for $0\leq \eps \leq 1/2$ we have by calculus that $(1-\eps)^2(1+6\eps)\geq 1$ and $(1-\eps)^2 \geq 1-2\eps$, 
% By an easy exercise, $(1-\eps)^2(1+6\eps) \geq 1$ because $\eps \in [0,\frac{1}{2}]$.
 and applying \cref{lem:zoom+} (with $k=n$ and $\lambda = 1-\eps$) yields that
  % $\Psi_*^{1-\varepsilon, n} \mathbb{P}_{(1+6 \eps) \beta, p} \geq_{\mathrm{st}} \mathbb{P}_{\beta, p^{\prime}}$, Hence,
\begin{align} \nonumber
    \p_{ (1+6\eps)\beta,p} \bra{ |K^\mathrm{max}_{2n}| \geq (1-2\eps) L^{2dn} } &\geq \p_{(1-\eps)^2(1+6\eps)\beta,p'} \bra{ |K^\mathrm{max}_n| \geq (1-\eps) L^{dn} }\\
    &\geq\p_{\beta,p'} \bra{ |K^\mathrm{max}_n| \geq (1-\eps) L^{dn} }
     \label{eq:alpha=d_key_lemma_second_bound}
\end{align}
 where 
 \[p':= \p_{(1+6\eps)\beta,p} \bra{ |K^\mathrm{max}_n| \geq (1-\eps) L^{dn}} \geq \p_{\beta,p} \bra{ |K^\mathrm{max}_n| \geq (1-\eps) L^{dn}}.\]
 If $\p_{\beta,p} \left(|K^\mathrm{max}_n|\geq (1-\eps) \abs{\Lambda_n} \right) \geq p$ then $p'\geq p$ and the claim follows immediately from \cref{eq:alpha=d_key_lemma_second_bound}.
% Now assume the hypothesis that
% \begin{equation} \label{eq:alpha=d_key_lemma_hyp}
%     \p_{\beta,p} \bra{ \abs{M_n} \geq (1-\eps) \abs{\Lambda_n} } \geq p.
% \end{equation}
% Since $(1+6\eps)\beta \geq \beta$, it follows that $p' \geq p$. Since $p' \geq p$, \cref{eq:alpha=d_key_lemma_second_bound} implies that
% \[
%     \p_{ (1+6\eps)\beta,p} \bra{ \abs{M_{2n}} \geq (1-2\eps) \abs{\Lambda_{2n}} } \geq \p_{\beta,p} \bra{ \abs{M_n} \geq (1-\eps) \abs{\Lambda_n} }.
% \]
% So by applying \cref{eq:alpha=d_key_lemma_hyp} for a second time,
% \[
%     \p_{ (1+6\eps)\beta,p} \bra{ \abs{M_{2n}} \geq (1-2\eps) \abs{\Lambda_{2n}} } \geq p. \qedhere
% \]
\end{proof}

It remains finally to prove \cref{lem:alpha=d_base}.

\begin{proof}[Proof of \cref{lem:alpha=d_base}]
Fix $\beta \geq 1$. Consider a mixed configuration $(\eta,\omega) \in \Omega$. As usual, we will abuse notation to think of $\eta$ and $\omega$ as subsets of $\bbH^d_L$ and $E$ when appropriate and recall that $F_k :=\bigcup_{z \in \mathbb{H}_L^d}E_k(z)$ is the set of all unordered pairs of distinct vertices of distance at most $L^k$.
 % Define $X := \{x \in \Lambda_{a\beta} : \eta_x = 1\}$. 
  Consider the configuration $(\eta',\omega') := \Psi^{L^{-2d},2}(\eta,\omega)$, which satisfies
\[
\eta'_x=\mathbbm{1}\Bigl(x=\pi^2(z) \text{ for some $z$ with $|K^\mathrm{max}_2(z)|\geq 1$}\Bigr)
=
\mathbbm{1}\Bigl(x=\pi^2(z) \text{ for some $z$ with $\eta \cap \Lambda_2(z) \neq \emptyset$}\Bigr).
% K^\mathrm{max}_k(z)|\geq \lambda |\Lambda_k|$}\Bigr)
\]
% and
% \begin{multline*}
% \omega'_{xy}=\mathbbm{1}\Bigl(x=\pi^k(z),y=\pi^k(w) \text{ for some $z,w$ with $\omega_{zw}=1$, }  \text{$z\in K^\mathrm{max}_k(z)$, and $w\in K^\mathrm{max}_k(w)$}\Bigr).
% \end{multline*}
   In order for the inequality $\abs{K^\mathrm{max}_{n_0}} \geq ( 1 -2 \delta) L^{dn_0}$ to hold, it suffices that the following four conditions all hold:
\begin{enumerate}
    \item $|\eta \cap \Lambda_{n_0}|\geq (1-2\delta) L^{dn_0}$;
     % for at least $( 1 -2 e^{-b \beta}) \abs{\Lambda_{a \beta}}$ vertices $x \in \Lambda_{a\beta}$;
    \item $\omega_{xy} = 1$ for all pairs of distinct vertices $x,y \in \eta \cap \Lambda_{n_0}$ with $\|x-y\|\leq L^{2}$;
    \item $\Lambda_{n_0 - 2} \subseteq \eta'$;
    \item For each $0\leq k \leq n_0-3$, the configuration $\Phi^{k} [\omega']$ contains every pair of unordered vertices of $\Lambda_{n_0-2-k}$ with distance exactly $L$.
     % for all $e \in F_1 \cap E_{n_0 - 2 - k}$ and all $k \in \{ 0,\dots, a\beta -3 \}$.
\end{enumerate}
Indeed, conditions 2-4 ensure that every vertex in $\eta\cap \Lambda_{n_0}$ is contained in a single $(\eta,\omega)$-cluster in $\Lambda_{n_0}$ while condition 1 ensures that this cluster has the required size.
 
 For each $1\leq i \leq 4$, let $\mathscr{A}_i$ be the event that the $i$th of these conditions holds. It suffices to prove that
\[
\P_{\frac{1}{2}\beta,1-\delta}(\mathscr{A}_i) = 1-o(\delta)
\]
  % if $\beta$ is sufficiently large then
% \[
% \P_{\frac{1}{2}\beta,1-\delta}(\mathscr{A}_i) \geq 1-\frac{\delta}{4}
% \]
for each $1\leq i \leq 4$ as $\beta \to \infty$, since this guarantees that $\P_{\frac{1}{2}\beta,1-\delta}(\cap_{i=1}^4\mathscr{A}_i)\geq 1-\delta$ when $\beta$ is sufficiently large. We bound each of these probabilities in order, and will use repeatedly that $\delta^{-2}\leq L^{dn_0} \leq L^d \delta^{-2}$ by definition of $\delta$ and $n_0$.

\begin{enumerate}
    \item 
For the event $\mathscr{A}_1$, the Chernoff bound
 \begin{multline*}
\P_{\frac{1}{2}\beta,1-\delta}(\mathscr{A}_1^c) = \mathbb{Q}_{1-\delta}\bigl(| \Lambda_{n_0}\setminus \eta|\geq 2\delta |\Lambda_{n_0}|\bigr) 
\leq e^{-2\lambda \delta L^{dn_0}} \E e^{\lambda | \Lambda_{n_0}\setminus \eta|} \\\leq e^{-2\lambda \delta L^{dn_0}} \left(1+  (e^{\lambda}-1)\delta \right)^{L^{dn_0}} \leq \exp\left[-\left(2\lambda -(e^{\lambda}-1) \right) \delta L^{dn_0}  \right]
 \end{multline*}
 holds for every $\lambda>0$, and taking $\lambda = \log 2$ yields that
 % Since $L^{dn_0} \geq \delta^{-2}$ it follows that
  % if $\beta$ is sufficiently large then
 \[
\P_{\frac{1}{2}\beta,1-\delta}(\mathscr{A}_1)  \geq 1-\left(\frac{e}{4}\right)^{\delta L^{dn_0}} \geq 1- \left(\frac{e}{4}\right)^{\delta^{-1}} =1-o(\delta)
 \]
 as required.
 \item
For the event $\mathscr{A}_2$, we have the union bound
 \begin{equation*} 
 % \label{eq:alpha=d_base_2}
   \P_{\frac{1}{2}\beta,1-\delta}(\mathscr{A}_2) = \p_{\frac{1}{2}\beta} \bra{F_2 \cap E_{n_0} \subseteq  \omega } \geq 1 - \abs{\Lambda_2}^2 \abs{\Lambda_{n_0-2}} e^{-\frac{1}{2} L^{-4d}\beta} =1-O\left(\delta^{-2}e^{-\frac{1}{2} L^{-4d}\beta}\right),
    % = 1-C_1 e^{ -\beta \sqbra{ \frac{1}{2 L^{4d}} - \tilde a } },
\end{equation*}
and since 
$e^{-\frac{1}{2} L^{-4d}\beta} = \delta^{\frac{1}{2}L^{5d}} \leq \delta^{16}$ it follows that 
\[\P_{\frac{1}{2}\beta,1-\delta}(\mathscr{A}_2) = 1-O(\delta^{14})=1-o(\delta)\]
as required.
\item 
For the event $\mathscr{A}_3$, it follows from \cref{lem:zoom+} and a union bound that
\[
\P_{\frac{1}{2}\beta,1-\delta}(\mathscr{A}_3) \geq 1-(1-p')L^{d(n_0-2)}
\qquad
\text{where}
% By \cref{lem:zoom+}, we know that $\Psi_*^{ \frac{1}{\abs{\Lambda_2}},2 } \p_{\frac{1}{2} \beta, 1-e^{-b\beta} } \geq_{\mathrm{st}} \p_{\beta',p'}$ where $\beta':= \frac{\beta}{2L^{4d}}$ and
\qquad
    p' := \p_{\frac{1}{2} \beta, 1-\delta } \bra{ |K^\mathrm{max}_2| \geq 1 } = 1 -\delta^{L^{2d}}
     % = 1- e ^{ -b L^{2d} \beta }.
\]
and hence that
\[
\P_{\frac{1}{2}\beta,1-\delta}(\mathscr{A}_3) \geq 1-\delta^{L^{2d}} L^{d(n_0-2)} = 1-O\left(\delta^{L^{2d}-2}\right) = 1-o(\delta)
% \left(\delta^{L^{2d}-2}\right)
\]
as required, where in the final estimate we used that $L^{2d}-2\geq 2>1$. (We zoomed out using $\Psi$ \emph{twice} precisely to make this step work; zooming out once would not be sufficient when $d=1$ and $L\in \{2,3\}$.)
% By a union bound,
% \begin{equation} \label{eq:alpha=d_base_3}
%     \mathbb Q_{p'} \bra{\forall x \in \Lambda_{a\beta - 2} : \;  \eta_x = 1} \geq 1-\abs{\Lambda_{a\beta-2}} (1-p') = 1 - C_2 e^{ - \beta\sqbra{ b L^{2d} - \tilde a } },
% \end{equation}
% where $C_2< \infty$ is independent of $\beta$. 

\item For the event $\mathscr{A}_4$, we will show that $\P_{\frac{1}{2}\beta,1-\delta}(\mathscr{A}_3\setminus \mathscr{A}_4)=o(\delta)$.  We have by \cref{lem:zoom+} that the conditional distribution of $\omega \cap E_{n_0-2}$ given $\mathscr{A}_3$ stochastically dominates $\P_{\beta'}$ where
$\beta'= L^{-4d} \beta$.
Since we also have by \cref{lem:zoom} that $\Phi_*^{k} \p_{\beta'} = \p_{\beta'}$ for all $k$, it follows by a union bound that
\begin{multline*}
\P_{\frac{1}{2}\beta,1-\delta}(\mathscr{A}_3\setminus \mathscr{A}_4) \leq \sum_{k=0}^{n_0-3} \P_{\beta'}(\omega \nsubseteq F_1 \cap E_{n_0-2-k})\\ \leq \sum_{k=0}^{n_0-3}|\Lambda_1|^2|\Lambda_{n_0-3-k}|e^{-L^{-2d}\beta'} = O(L^{dn_0}e^{-L^{-6d}\beta}) =O\left(\delta^{L^{3d}-2}\right) =O(\delta^6)=o(\delta).
\end{multline*}
Since we also have that $\P_{\frac{1}{2}\beta,1-\delta}(\mathscr{A}_3)=1-o(\delta)$, it follows that $\P_{\frac{1}{2}\beta,1-\delta}(\mathscr{A}_4)=1-o(\delta)$ as required.
% , for each $k \in \{ 0,\dots,a\beta -3 \}$,
% \[
%     \p_{\beta'} \bra{ \forall e \in F_1 \cap E_{a\beta - 2 -k} : \; \Phi^{k}[\omega](e) = 1 } \geq 1 - \abs{\Lambda_1}^2 \abs{\Lambda_{a\beta-3-k}} e^{ - \frac{\beta'}{L^{2d}} } = 1 - C_3 L^{-dk} e^{ - \beta \sqbra{ \frac{1}{2 L^{6d}} - \tilde a } },
% \]
% where $C_3 < \infty$ is independent of $\beta$ and $k$.
% Taking a further union bound over $k$, it follows that
% \begin{multline} \label{eq:alpha=d_base_4_summed}
%     \p_{\beta'} \bra{ \forall k \in \{ 0,\dots,a\beta -3 \} \;\forall e \in F_1 \cap E_{a\beta - 2 -k} : \; \Phi^{k}[\omega](e) = 1 } \\\geq 1 - C_3 e^{ - \beta \sqbra{ \frac{1}{2 L^{6d}} - \tilde a } } \sum_{k=0}^{ a \beta - 3 } L^{-dk}
%     \geq 1 - C_4 e^{ - \beta \sqbra{ \frac{1}{2 L^{6d}} - \tilde a } },
% \end{multline}
% where $C_4 := C_3 \sum_{k =0}^{\infty} L^{-dk} = \frac{C_3}{1-L^{-d}} < \infty$ is independent of $\beta$.
\end{enumerate}
This concludes the proof. \qedhere
\end{proof}

\subsection{Upper bounds}
\label{sec:alpha=d_upper}

We conclude the paper with a short proof of the upper bound of \cref{thm:main} in the case $\alpha =d$, recovering a result of \cite{georgakopoulos2020percolation}.

\begin{prop} \label{prop:alpha=d_ub}
If $\alpha=d$ then there exists $C=C(d,L) <\infty$ such that 
    $\chi(\beta) \leq e^{ e^{C \beta} }$
for every $\beta \geq 1$.
\end{prop}

We will prove \cref{prop:alpha=d_ub} by proving an equivalent upper bound on the \emph{correlation length}  $\xi(\beta)$ as defined in \cite[Section 4]{MR4462652}.
Following Duminil-Copin and Tassion \cite{duminil2015new}, for each  $\beta\geq 0$, and finite subset $S \subseteq \bbH^d_L$  containing the origin we consider the quantity
\[
\phi_\beta(S) =\phi_\beta(S,0) := \sum_{y\notin S}\sum_{x\in S} \Bigl(1-e^{-\beta \|x-y\|^{-d-\alpha}}\Bigr)\P_\beta(0 \xleftrightarrow{S} x),
\]
where we write $\{0 \xleftrightarrow{S} x\}$ to mean that $0$ and $x$ are connected by an open path all of whose vertices belong to $S$. It is a straightforward consequence of the BK inequality as explained in \cite[Lemma 4.2]{MR4462652} that
% 
% \begin{lemma}
% \label{lem:phi_beta_S}
% Let $G=(V,E,J)$ be a countable weighted graph and let $S\subseteq \Lambda$ be finite subsets of $V$. Then
\[
\sum_{x\in S'} \P_\beta(0 \xleftrightarrow{S'} x) \leq \sum_{x\in S} \P_\beta(0 \xleftrightarrow{S} x) + \phi_\beta(S) \cdot \sup_{u\in S'} \sum_{x\in S'} \P_\beta(u \xleftrightarrow{S'} x)
\]
for every $\beta\geq 0$ and every pair of finite sets $S \subseteq S' \subseteq \bbH^d_L$. As such, if $\phi_\beta(S)<1$ then we may take the limit as $S'$ exhausts $\bbH^d_L$ to obtain that
\[
\sum_{x\in S} \P_\beta(0 \xleftrightarrow{S} x) \leq \chi(\beta)\leq \frac{1}{1-\phi_\beta(S)}\sum_{x\in S} \P_\beta(0 \xleftrightarrow{S} x).
\]
% \end{lemma}
For each $n \geq 0$ we define
\[
\beta_n = \sup\left\{\beta \geq 0 : \phi_\beta(\Lambda_n,0) \leq \frac{1}{2} \right\} \qquad \text{ and } \qquad \beta^*_n = \max_{0 \leq m \leq n} \beta_m,
\]
so that $(\beta_n^*)_{n\geq 0}$ is a non-decreasing sequence.
For each $0\leq \beta<\beta_c$ we define the \textbf{correlation length} $\xi(\beta)$ by
\begin{equation}
\label{eq:correlation_length_definition}
\xi(\beta) = L^{n(\beta)} \qquad \text{ where } \qquad  n(\beta)=\inf\{n\geq 0 : \beta \leq \beta_n^*\},
\end{equation}
which has the property that the global susceptibility $\chi(\beta)$ is within a factor of two of the expected number of points that are connected to $0$ within the ball of radius $\xi(\beta)$:
\begin{equation}
\label{eq:correlation_length_susceptibility}
\sum_{x\in \Lambda_{n(\beta)}} \P_\beta(0\xleftrightarrow{\Lambda_{n(\beta)}} x ) \leq \chi(\beta) \leq 2 \sum_{x\in \Lambda_{n(\beta)}} \P_\beta(0\xleftrightarrow{\Lambda_{n(\beta)}} x ).
\end{equation}
(Note that this estimate holds for every $\alpha>0$.)
Since the right hand side is trivially at most $2|\Lambda_{n(\beta)}|=2 \xi(\beta)^d$, \cref{prop:alpha=d_ub} follows from \eqref{eq:correlation_length_susceptibility} and the following proposition.

\begin{prop}
\label{prop:alpha=d_correlation_length} If $\alpha =d$ then there exist constants $c=c(d,L)>0$ and $C=C(d,L)<\infty$ such that the correlation length satisfies 
$e^{e^{c\beta}}\leq \xi(\beta) \leq e^{e^{C\beta}}$
 for every $\beta \geq 1$.
\end{prop}

\begin{proof}
The lower bound follows from \cref{prop:alpha=d_lb} since $\chi(\beta)\leq 2\xi(\beta)^d$; it remains to prove the upper bound. We begin by bounding $\P_\beta(0\leftrightarrow \Lambda_n^c)$ for appropriately large $n$ using an exploration argument. 
Define a random sequence $(n_i)_{i\geq 0}$ by setting $n_0=0$ and recursively setting $n_{i+1}$ to be maximal such that there is an open edge connecting $\Lambda_{n_i}$ to $\Lambda_{n_{i+1}}\setminus \Lambda_{n_{i+1}-1}$, taking $n_{i+1}=n_i$ if there are no open edges incident to $\Lambda_{n_i}$. We define $\tau$ to be the minimal $i$ such that $n_{i+1}=n_i$, so that $0\leftrightarrow \Lambda_n^c$ only if $n_\tau>n$. For each $i\geq 0$ let $\mathcal{F}_i$ be the $\sigma$-algebra generated by $n_0,\ldots,n_i$. Since we can compute $n_0,\ldots,n_i$ in such a way that we only reveal edges with at least one endpoint in $\Lambda_{n_i}$ and any revealed edge with an endpoint in $\Lambda_{n_i}^c$ is closed, we have that
\begin{align*}
\P_\beta(n_{i+1}> n_i+k \mid \mathcal{F}_i) &\leq \P_\beta(\text{there is an open edge connecting $\Lambda_{n_i}$ and $\Lambda_{n_i+k}^c$})\\
&=\P_\beta(\text{there is an open edge connecting $\Lambda_{0}$ and $\Lambda_{k}^c$})
\end{align*}
almost surely for each $i,k\geq 0$, where the final equality follows from \cref{lem:zoom} (where $L^{d-\alpha}\beta=\beta$ since $\alpha=d$).
Letting $X_1,X_2,\ldots$ be i.i.d.\ random variables with distribution
\[
\P(X_1 > k) = \P_\beta(\text{there is an open edge connecting $\Lambda_{0}$ and $\Lambda_{k}^c$})
\]
and letting $T=\min\{i:X_i =0\}$, it follows that
\[
\P(n_\tau \geq n) \leq \P\left(\sum_{i=1}^T X_i \geq n\right).
\]
Now, we can compute that
\begin{align*}
\P(X_1 > k)
  &= \P_\beta(\text{there is an open edge connecting $\Lambda_{0}$ and $\Lambda_{k}^c$}) 
  \\&= 1-\exp\left[-\beta\sum_{\ell=k+1}^\infty  L^{-2d\ell} (L^{d\ell}-L^{d(\ell-1)}) \right]
  = 1-\exp\left[-\beta L^{-d(k+1)}\right]
  % \sum_{\ell=k+1}  L^{-d\ell} \right]
\end{align*}
so that 
$\P(X_1=0)=\exp\left[-L^{-d}\beta \right]$
%  Moreover, letting $k_0=k_0(\beta)$ be the smallest integer such that $L^{dk}\geq \beta$, we have by calculus that there exists a universal constant $\alpha<1$ such that
% \[
% \P(X_1 > k +1 \mid X_1 >k) = \frac{1-\exp\left[-\beta L^{-d(k+2)}\right]}{1-\exp\left[-\beta L^{-d(k+1)}\right]} \leq \alpha
% \]
% for every $k\geq k_0$, so that $X_1$ is stochastically dominated by the sum of $k_0$ with a geometric random variable of parameter $\alpha$ and hence in particular that $\E X_1 =O(k_0)=O(\log \beta)$. 
and
\[
\E X_1 = \sum_{k=0}^\infty \P(X_1 >k) = \sum_{k=0}^\infty 1-\exp\left[-\beta L^{-d(k+1)}\right] \leq \beta \sum_{k=0}^\infty L^{-d(k+1)} = \frac{L^{-d}}{1-L^{-d}}\beta.
\] 
Thus, we have by Markov's inequality that
\[
\P_\beta(0\leftrightarrow \Lambda_n^c) \leq \P_\beta(n_\tau \geq n) \leq \P(T \geq t) + \P\left(\sum_{i=1}^t X_i \geq n\right) \leq (1-e^{-L^{-d}\beta})^t + \frac{L^{-d}}{1-L^{-d}} \frac{\beta t}{n}
\]
for every $n,t\geq 1$. Taking $t=\lceil \beta^{-3} n \rceil$, we deduce that there exist constants $C_1$ and $C_2$ such that if $n \geq C_1 \beta^4 e^{L^{-d} \beta} $ then
\begin{equation}
\label{eq:escape_large_ball}
\P_\beta(0\leftrightarrow \Lambda_n^c) \leq (1-e^{-L^{-d}\beta})^{\lceil \beta^{-3} n \rceil} + O(\beta^{-2}) \leq  C_2 \beta^{-2}.
\end{equation}
To complete the proof, we note (using that $1-e^{-\beta \|x-y\|^{-2d}} \leq \beta \|x-y\|^{-2d}$) that there exists a constant $C_3$ such that
\[
\phi_\beta(\Lambda_n) \leq \sum_{y\in \Lambda_n^c}\sum_{x\in \Lambda_n} \beta \|x-y\|^{-2d} \P_\beta(0 \xleftrightarrow{\Lambda_n} x) \leq C_3 \beta L^{-dn} \sum_{x\in \Lambda_n} \P_\beta(0 \xleftrightarrow{\Lambda_n} x)
\]
for every $n\geq 1$ and that
\[
 \sum_{x\in \Lambda_n} \P_\beta(0 \xleftrightarrow{\Lambda_n} x) \leq 1+ \sum_{k=0}^n L^{d(k+1)}\P_\beta(0\leftrightarrow \Lambda_k^c),  
\]
so that there exist constants $C_4$ and $C_5$ such that if $n \geq 2C_1 \beta^4 e^{L^{-d} \beta}$ then 
\begin{align*}
\sum_{x\in \Lambda_n} \P_\beta(0 \xleftrightarrow{\Lambda_n} x) &\leq
1+\sum_{k=0}^{\lfloor n/2 \rfloor} L^{d(k+1)} +  \sum_{k=\lceil n/2 \rceil}^{n} L^{d(k+1)} C_2 \beta^{-2} \\&\leq 1+
 \sqrt{L^{d(n+1)}}+C_4 \beta^{-2} L^{dn} \leq C_5 \beta^{-2} L^{dn},
\end{align*}
where we used that $\lceil n/2 \rceil \geq C_1 \beta^4 e^{L^{-d} \beta}$ to apply \eqref{eq:escape_large_ball} in the first inequality.
It follows that if $n\geq 2C_1 \beta^4 e^{L^{-d} \beta}$ then
$\phi_\beta(\Lambda_n) \leq C_3 C_5 \beta^{-1}$,
which is less than $1/2$ when $\beta$ is sufficiently large. This implies that 
$n(\beta) = O\left(\beta^4 e^{L^{-d} \beta}\right)=e^{O(\beta)}$
and hence that $\xi(\beta)=e^{e^{O(\beta)}}$ as required. \qedhere
\end{proof}

\begin{proof}[Proof of \cref{prop:alpha=d_ub}]
The claim follows immediately from \cref{prop:alpha=d_correlation_length} and the inequality $\chi(\beta)\leq 2\xi(\beta)^d$.
\end{proof}

\begin{proof}[Proof of \cref{thm:main}]
The case $\alpha>d$ follows from \cref{lem:alpha>d_ub,prop:alpha>d_lb} while the case $\alpha=d$ follows from \cref{prop:alpha=d_lb,prop:alpha=d_ub}.
\end{proof}

\subsection*{Acknowledgements} This work was carried out as part of Caltech's Summer Undergraduate Research Fellowship (SURF) program 2022, during which JK was mentored by PE and TH. During the research, JK was also supported by an NSERC USRA. We thank Louigi Addario-Berry and Johannes B\"aumler for helpful comments on a draft.

\setstretch{1}
\footnotesize{
\bibliographystyle{abbrv}
\bibliography{unimodularthesis.bib} 

\begin{thebibliography}{10}

\bibitem{MR894398}
M.~Aizenman, D.~J. Barsky, and R.~Fern\'andez.
\newblock The phase transition in a general class of {I}sing-type models is
  sharp.
\newblock {\em J. Statist. Phys.}, 47(3-4):343--374, 1987.

\bibitem{aizenman1988discontinuity}
M.~Aizenman, J.~Chayes, L.~Chayes, and C.~Newman.
\newblock Discontinuity of the magnetization in one-dimensional $1/|x- y|^2$
  {I}sing and {P}otts models.
\newblock {\em Journal of Statistical Physics}, 50(1):1--40, 1988.

\bibitem{MR3969983}
R.~Bauerschmidt, D.~C. Brydges, and G.~Slade.
\newblock {\em Introduction to a renormalisation group method}, volume 2242 of
  {\em Lecture Notes in Mathematics}.
\newblock Springer, Singapore, 2019.

\bibitem{baumler2022behavior}
J.~B{\"a}umler.
\newblock Behavior of the distance exponent for $1/|x-y|^{2d}$ long-range
  percolation.
\newblock {\em arXiv preprint arXiv:2208.04793}, 2022.

\bibitem{baumler2022distances}
J.~B{\"a}umler.
\newblock Distances in $1/|x- y|^{2d}$ percolation models for all dimensions.
\newblock {\em arXiv preprint arXiv:2208.04800}, 2022.

\bibitem{baumler2022isoperimetric}
J.~B{\"a}umler and N.~Berger.
\newblock Isoperimetric lower bounds for critical exponents for long-range
  percolation.
\newblock {\em arXiv preprint arXiv:2204.12410}, 2022.

\bibitem{bleher1987critical}
P.~Bleher and P.~Major.
\newblock Critical phenomena and universal exponents in statistical physics. on
  {D}yson's hierarchical model.
\newblock {\em The Annals of Probability}, pages 431--477, 1987.

\bibitem{dawson2013percolation}
D.~Dawson and L.~Gorostiza.
\newblock Percolation in an ultrametric space.
\newblock {\em Electronic Journal of Probability}, 18:1--26, 2013.

\bibitem{ding2013distances}
J.~Ding and A.~Sly.
\newblock Distances in critical long range percolation.
\newblock {\em arXiv preprint arXiv:1303.3995}, 2013.

\bibitem{dragovich2017p}
B.~Dragovich, A.~Y. Khrennikov, S.~Kozyrev, I.~Volovich, and E.~Zelenov.
\newblock p-adic mathematical physics: the first 30 years.
\newblock {\em P-Adic numbers, ultrametric analysis and applications},
  9(2):87--121, 2017.

\bibitem{dragovich2009p}
B.~Dragovich, A.~Y. Khrennikov, S.~V. Kozyrev, and I.~V. Volovich.
\newblock On p-adic mathematical physics.
\newblock {\em P-Adic Numbers, Ultrametric Analysis, and Applications},
  1(1):1--17, 2009.

\bibitem{duminil2020long}
H.~Duminil-Copin, C.~Garban, and V.~Tassion.
\newblock Long-range models in 1d revisited.
\newblock {\em arXiv preprint arXiv:2011.04642}, 2020.

\bibitem{duminil2015new}
H.~Duminil-Copin and V.~Tassion.
\newblock A new proof of the sharpness of the phase transition for {B}ernoulli
  percolation and the {I}sing model.
\newblock {\em Comm. Math. Phys.}, 343(2):725--745, 2016.

\bibitem{dyson1969existence}
F.~J. Dyson.
\newblock Existence of a phase-transition in a one-dimensional {I}sing
  ferromagnet.
\newblock {\em Communications in Mathematical Physics}, 12(2):91--107, 1969.

\bibitem{edwards1988generalization}
R.~G. Edwards and A.~D. Sokal.
\newblock Generalization of the fortuin-kasteleyn-swendsen-wang representation
  and monte carlo algorithm.
\newblock {\em Physical review D}, 38(6):2009, 1988.

\bibitem{georgakopoulos2020percolation}
A.~Georgakopoulos and J.~Haslegrave.
\newblock Percolation on an infinitely generated group.
\newblock {\em Combinatorics, Probability and Computing}, 29(4):587--615, 2020.

\bibitem{MR2243761}
G.~Grimmett.
\newblock {\em The random-cluster model}, volume 333 of {\em Grundlehren der
  mathematischen Wissenschaften [Fundamental Principles of Mathematical
  Sciences]}.
\newblock Springer-Verlag, Berlin, 2006.

\bibitem{1901.10363}
T.~Hutchcroft.
\newblock New critical exponent inequalities for percolation and the random
  cluster model.
\newblock {\em Probab. Math. Phys.}, 1(1):147--165, 2020.

\bibitem{hutchcroft2021critical}
T.~Hutchcroft.
\newblock The critical two-point function for long-range percolation on the
  hierarchical lattice.
\newblock {\em arXiv preprint arXiv:2103.17013}, 2021.

\bibitem{hutchcroft2022critical}
T.~Hutchcroft.
\newblock Critical cluster volumes in hierarchical percolation.
\newblock {\em arXiv preprint arXiv:2211.05686}, 2022.

\bibitem{MR4462652}
T.~Hutchcroft.
\newblock On the derivation of mean-field percolation critical exponents from
  the triangle condition.
\newblock {\em J. Stat. Phys.}, 189(1):Paper No. 6, 33, 2022.

\bibitem{MR4504407}
T.~Hutchcroft.
\newblock Sharp hierarchical upper bounds on the critical two-point function
  for long-range percolation on {$\Bbb{Z}^d$}.
\newblock {\em J. Math. Phys.}, 63(11):Paper No. 113301, 18, 2022.

\bibitem{koval2012long}
V.~Koval, R.~Meester, and P.~Trapman.
\newblock Long-range percolation on the hierarchical lattice.
\newblock {\em Electronic Journal of Probability}, 17:1--21, 2012.

\bibitem{ouboter2016stochastic}
T.~Ouboter, R.~Meester, and P.~Trapman.
\newblock Stochastic {SIR} epidemics in a population with households and
  schools.
\newblock {\em Journal of Mathematical Biology}, 72(5):1177--1193, 2016.

\bibitem{sawyer1983isolation}
S.~Sawyer and J.~Felsenstein.
\newblock Isolation by distance in a hierarchically clustered population.
\newblock {\em Journal of Applied Probability}, 20(1):1--10, 1983.

\end{thebibliography}
}

\end{document}